\documentclass{amsart}

\usepackage{amssymb}

\usepackage{amsmath}
\usepackage[title]{appendix}
\usepackage{graphicx, enumerate, url}
\usepackage{booktabs}
\usepackage{subcaption}
\usepackage{multirow,bigdelim}

\usepackage{stmaryrd}
\usepackage{algorithm}
\usepackage{algpseudocode}

\usepackage[mathscr]{eucal}

\usepackage{hyperref}
\usepackage{cleveref}

\newtheorem{definition}{Definition}
\newtheorem{remark}{Remark}
\newtheorem{proposition}{Proposition}
\newtheorem{lemma}{Lemma}
\newtheorem{theorem}{Theorem}

\theoremstyle{definition}

\newcommand{\sfT}{\mathsf{T}}

\newcommand{\ttT}{\mathtt{T}}
\newcommand{\ttX}{\mathtt{X}}
\newcommand{\ttY}{\mathtt{Y}}
\newcommand{\ttZ}{\mathtt{Z}}
\newcommand{\ttW}{\mathtt{W}}

\newcommand{\ttK}{\mathtt{K}}
\newcommand{\ttM}{\mathtt{M}}

\newcommand{\scrK}{\mathscr{K}}
\newcommand{\scrF}{\mathscr{F}}

\newcommand{\vd}{\vec{d}}

\newcommand*\qs{\textsc{\char13}s }

\newcommand{\sfS}{\mathsf{S}}
\newcommand{\ttF}{\mathsf{F}}

\newcommand{\bttY}{\overbar{\mathtt{Y}}}
\newcommand{\bttZ}{\overbar{\mathtt{Z}}}

\newcommand{\R}{\mathbb{R}}

\newcommand{\rD}{\mathrm{D}}
\newcommand{\rG}{\mathcal{G}}
\newcommand{\trG}{\tilde{\mathcal{G}}}
\newcommand{\trH}{\tilde{\mathcal{H}}}
\newcommand{\trK}{\tilde{\mathcal{K}}}

\newcommand{\fg}{\mathfrak{g}}
\newcommand{\fq}{\mathfrak{q}}
\newcommand{\fk}{\mathfrak{k}}

\newcommand{\tfm}{\tilde{\fm}}
\newcommand{\tfg}{\tilde{\fg}}
\newcommand{\tfh}{\tilde{\mathfrak{h}}}
\newcommand{\tfk}{\tilde{\fk}}

\newcommand{\tfmfq}{\tilde{\fm}_{\fqh}}

\newcommand{\hrB}{\hat{\rB}}

\newcommand{\ttpa}{\mathcal{P}_\mathfrak{a}}
\newcommand{\ttpap}{\mathcal{P}_{\faperp}}
\newcommand{\ttpm}{\mathcal{P}_\mathfrak{m}}
\newcommand{\faperp}{\mathfrak{a}_{_{\perp}}}
\newcommand{\fvperp}{\mathfrak{v}_{_{\perp}}}
\newcommand{\fm}{\mathfrak{m}}
\newcommand{\fv}{\mathfrak{v}}

\newcommand{\rH}{\mathcal{H}}

\newcommand{\fqh}{\bar{\fq}}

\newcommand{\ttA}{\mathtt{A}}
\newcommand{\ttB}{\mathtt{B}}
\newcommand{\ttC}{\mathtt{C}}

\newcommand{\ttV}{\mathtt{V}}

\newcommand{\rK}{\mathcal{K}}
\newcommand{\rA}{\mathcal{A}}

\newcommand{\sym}{\mathrm{sym}}
\newcommand{\fsym}{\mathrm{symf}}

\newcommand{\asym}{\mathrm{skew}}

\newcommand{\C}{\mathbb{C}}

\newcommand{\V}{\mathbb{V}}

\newcommand{\cV}{\mathcal{V}}

\newcommand{\cW}{\mathcal{W}}

\newcommand{\cD}{\mathcal{D}}

\newcommand{\cE}{\mathcal{E}}

\newcommand{\Herm}{\mathrm{Sym}}
\newcommand{\AHerm}{\mathrm{Skew}}
\newcommand{\St}[2]{\mathrm{St}_{#1, #2}}
\newcommand{\Gr}[2]{\mathrm{Gr}_{#1, #2}}

\newcommand{\Spl}{\mathrm{Sp}}

\newcommand{\cH}{\mathcal{H}}

\newcommand{\cI}{\mathcal{I}}

\newcommand{\cF}{\mathcal{F}}
\newcommand{\cP}{\mathcal{P}}

\newcommand{\cK}{\mathcal{K}}
\newcommand{\cM}{\mathcal{M}}
\newcommand{\overbar}[1]{\mkern 1.5mu\overline{\mkern-1.5mu#1\mkern-1.5mu}\mkern 1.5mu}
\newcommand{\bcM}{\overbar{\mathcal{M}}}

\newcommand{\gl}{\mathfrak{gl}}
\newcommand{\oo}{\mathfrak{so}}

\newcommand{\SO}{\mathrm{SO}}

\newcommand{\cS}{\mathcal{S}}

\newcommand{\lb}{\llbracket}
\newcommand{\rb}{\rrbracket}

\newcommand{\rI}{\rm{I}}
\newcommand{\rB}{\mathcal{B}}

\newcommand{\bX}{\overbar{X}}

\newcommand{\bA}{A}

\newcommand{\bxi}{\bar{\xi}}
\newcommand{\bareta}{\bar{\eta}}

\newcommand{\bY}{Y}

\newcommand{\rP}{\mathrm{P}}

\newcommand{\sfg}{\mathbf{g}}

\newcommand{\fa}{\mathfrak{a}}

\newcommand{\vbeta}{\vec{\beta}}
\newcommand{\Yperp}{Y_{_{\perp}}}

\newcommand{\dIperp}{\dI_{_{\perp}}}

\DeclareMathOperator{\diag}{diag}

\DeclareMathOperator{\ad}{ad}
\DeclareMathOperator{\Ad}{Ad}
\DeclareMathOperator{\Tr}{Tr}

\DeclareMathOperator{\GL}{GL}

\DeclareMathOperator{\OO}{O}

\DeclareMathOperator{\Flag}{Flag}

\DeclareMathOperator{\dI}{I}

\DeclareMathOperator{\rC}{C}

\DeclareMathOperator{\expm}{expm}
\DeclareMathOperator{\expa}{expv}

\begin{document}

\ifpdf
  \DeclareGraphicsExtensions{.eps,.pdf,.png,.jpg}
\else
  \DeclareGraphicsExtensions{.eps}
\fi





\title{Parallel transport on matrix manifolds and Exponential Action}

\author{Du Nguyen}
\address{Independent, Darien, CT}
\email{nguyendu@post.harvard.edu}

\author{Stefan Sommer}
\address{University of Copenhagen}
\email{sommer@di.ku.dk}

\begin{abstract}We express parallel transport for several common matrix Lie groups with a family of pseudo-Riemannian metrics in terms of matrix exponential and exponential actions. The metrics are constructed from a deformation of a bi-invariant metric and are naturally reductive. There is a similar picture for homogeneous spaces when taking quotients satisfying a general condition. In particular, for a Stiefel manifold of orthogonal matrices of size $n\times d$, we give an expression for parallel transport along a geodesic from time zero to $t$, that could be computed with time complexity of $O(n d^2)$ for small $t$, and of $O(td^3)$ for large $t$, contributing a step in a long-standing open problem in matrix manifolds. A similar result holds for {\it flag manifolds} with the canonical metric. We also show the parallel transport formulas for the {\it general linear group} and the {\it special orthogonal group} under these metrics.
\end{abstract}


\maketitle

\section{Introduction}In recent years, Riemannian geometry has found many applications in optimization, image, vision, and data modeling with constraints and symmetry. A fundamental concept in geometry is {\it parallel transport}. While a tangent vector in Euclidean space can be transported between different points with almost no additional calculation, moving a tangent vector between points on a Riemannian manifold is more challenging computationally. Parallel transport allows one to move tangent vectors along curves in a manifold in a way consistent with geometry, which is important in clustering problems, and in algorithms that require accounting for past increments. For example, the well-known optimization methods L-BFGS or conjugate gradient on an Euclidean space keep a history of past movements and use them to decide the next moves. These increments are represented as tangent vectors at different points in the iteration and must be transported to the current iterative point for aggregation. Similar issues arise in computer vision.

In the now classical paper \cite{Edelman_1999} twenty-five years ago, the authors twice mentioned that there was no known $O(nd^2)$ algorithm for parallel transport on the Stiefel manifold $\St{n}{d}$. To the best of our knowledge, this problem remains unsolved. In this paper, we contribute toward a solution for this problem with the help of \emph{exponential action}. We show with an appropriate transformation, a class of metrics on matrix manifolds have geodesics computable by matrix exponential, while the parallel transport equation could be reduced to a linear ODE with constant coefficients, and thus could be solved by exponential action. For the case of the Stiefel manifolds, we show any tangent vector could be split into two parts, one in a submanifold similar to a great circle, where we can use exponential action. The remaining part, normal to a vector subspace containing the geodesic, can be transported using matrix exponential. We also show that the formula for the Stiefel manifold is a special case for pseudo-Riemannian metrics in a family with closed-form geodesic and parallel transport. Besides the state-of-the-art algorithm in the Stiefel case, we hope our results will provide insight into the properties of parallel transport.

\subsection{Contribution}
We give an explicit formula \cref{eq:stieftran} with evaluation cost $O(C_1nd^2+C_2td^3)$ for the parallel transport along a geodesic segment of length $t$ on the Stiefel manifold $\St{n}{d}$ of orthogonal matrices of size $n\times d$ ($d < n$) \cref{eq:stieftran}  for constants $C_1$ and $C_2$. The cost $O(C_2td^3)$ comes from the evaluation of an \emph{exponential action}. This time-dependent cost is less significant for a shorter time, which is the typical situation in optimization. Thus, in this case, the remaining cost $O(C_1nd^2)$ becomes dominant. For a bounded $t$, the algorithm could also be considered of time complexity $O(nd^2)$, and since Stiefel manifolds are compact, there is an upper bound to the distance between two points, thus, the algorithm is practically $O(nd^2)$ for reasonable time $t$\footnote{We thank a reviewer for this observation.}. While the exponential action could be expensive, we propose a custom implementation based on a norm estimate that speeds up the evaluation significantly. We also hope the closed-form formula provides new insights into parallel transports, leading to new transport algorithms and applications.

The framework to derive this formula also applies to manifolds related to Stiefel manifolds, giving us a new parallel transport formula for flag manifolds with the canonical metric and recovering the known formula for Grassmann manifolds.

As explained, this result is based on the observation that transport in directions tangent to a Stiefel submanifold (similar to a great circle) could be done efficiently by solving an equivalent problem that is reduced in dimension and is called \emph{baby problem} in \cite[chapter 5]{Rentmee}, while transport in directions orthogonal to this submanifold could be done by matrix exponential. This observation may apply to other families of manifolds.

The \emph{baby problem} is solved by lifting to a matrix Lie group. This solution generalizes to matrix Lie groups and homogeneous spaces equipped with a \emph{deformation metric} \cite{Cheeger1973,Jensen,DAZ}. In particular, given a Lie group $\rG$ with a bi-invariant pseudo-Riemannian metric represented by a symmetric bilinear form on its Lie algebra $\fg$ and a Lie subalgebra $\fa\subset\fg$ such that the bilinear form is nondegenerate when restricted to $\fa$, if we replace the metric on $\fa$ with another bi-invariant metric then the resulting metric is naturally reductive. The geodesic equation and geodesics are explicit, while parallel transport could be computed by exponential action. The replacement metric on $\fa$ is often a scalar multiple of the original metric. In some examples, we can change the sign of the metric on $\fa$, modifying the bi-invariant pseudo-Riemannian metric on $\fg$ to a naturally reductive Riemannian one, as in the case of $\GL(n)$ in \cite{VAV} and \cref{sec:GL}.

For a homogeneous manifold $\rG/\rK$ with $\rG$ equipped with a deformation metric, we give a general condition in \cref{theo:hornatural} for the manifold $\rG/\rK$ to be also naturally reductive, with geodesics and parallel transport computable by exponential action.

With this result, we show explicit parallel transport formulas for this family of metrics for the {\it general linear groups} and the {\it special orthogonal groups}. While not discussed here, the results also apply to \emph{symplectic groups} \cite{Fiori}. In particular, the Euclidean embedded metric of  Stiefel manifolds and $\GL(n)$ with the Frobenius metric \cite{Edelman_1999,VAV} are naturally reductive. While instances of this result are already known in the literature \cite{DAZ,VMSil,ExtCurveStiefel}, 
the explicit ODE for parallel transport, based on \cite{SmithThesis,ExtCurveStiefel} and the observation that certain tractable Riemannian metrics on non compact groups could be considered a deformation of the pseudo-Riemannian trace form is novel.

The nondegeneracy of the trace form as a bi-invariant metric is guaranteed if we restrict to Lie subalgebras that are equal to their own transpose, such as $\oo(n)$ or the full $n\times n$ matrix algebra, as opposed to the algebra of upper triangular matrices. 


\subsection{Related literature and State-of-the-Art}Closed-form geodesics formulas for the general linear group are in \cite{VAV,MartinNeff,Barp2019}. In the original paper \cite{Edelman_1999}, Edelman et. al. gave an equation (equation 2.27) for parallel transport for the Stiefel manifold with the {\it canonical metric}. The differential equation is linear but even with exponential action, the cost is $O(n^3)$. They also mentioned that this formula is a special case for reductive spaces and also gave an $O(nd^2)$ algorithm for the Grassmann transport. The paper referenced the thesis \cite{SmithThesis}, which gave a parallel transport formula for a quotient $\rG/\rH$ in its Proposition 2.12. The proof assumed a bi-invariant metric on $\rG$. When the metric is constructed from a deformation, (see \cref{theo:main,theo:hornatural}), which is essentially a quotient metric on a larger group $\trG$ with a bi-invariant metric, we can apply this result to compute the parallel transport. Exploiting a special property of Stiefel and flag manifolds to split the transported vector as noted previously, we avoid the $O(n^3)$ cost.


In \cite{HuperFlor}, the authors give a formula for parallel transport for Stiefel manifolds in the Euclidean embedded metric ($\alpha=1$ in \cref{eq:StiefelMetric}). The formula is in terms of a function $F$ satisfying a matrix ODE, depending on the equation of the curve, and requires vectorization. The geodesic is calculated with $O(n^3)$ cost for a given time. Vector transport \cite[chapter 8]{AMS_book}, often used in Riemannian optimization, is another method to transport tangent vectors on manifolds. They are often less expensive computationally, although they are generally not isometries (but there are remedies \cite{SatoRiemannianCG} in Riemannian optimization). We are not aware of other progress in parallel transport for Stiefel manifolds. Numerically, in \cite{ladder,LoPenn,LoAyPenn}, the authors introduce numerical algorithms to solve transport problems, in particular, the ladder schemes.
As closed-form formulas are not available most of the time, they are of great importance.

The family of metrics considered here could be considered a deformation of a pseudo-Riemannian metric (also called semi-Riemannian metric) to a Riemannian metric, see \cite{GZ2000}, the deformation is often called the Cheeger deformation, although it has also been used by other authors \cite{Jensen,Sagle1970}. The deformation comes from a (pseudo-Riemannian) submersion and is often used to produce metrics with desired curvature properties.  In \cite[Theorem 1]{DAZ}, the authors analyzed naturally reductive space arising from this construction for compact groups. Our \cref{theo:main} could be considered a partial extension of their construction to the pseudo-Riemannian case. The article \cite{Gordon1985} considers the non-compact spaces from a Lie theoretic approach, which we complement by offering an explicit metric construction and geometric formulas.

In \cite{VMSil}, the authors analyze sub-Riemannian geodesics arising from this family and note that this type of deformation produces closed-form geodesic. The resulting induced metrics for Stiefel manifolds are studied in \cite{ExtCurveStiefel}. In \cite{NguyenLie}, one of the authors computed the connection and curvature for this family in the case of a Riemannian bi-invariant metric. Flag manifolds are studied by many authors, where Riemannian Hessian, exponential map, logarithm, and the Levi-Civita connection are studied in \cite{YeLim,Pennec,Szwagier,NguyenOperator,NguyenGeodesic}.

The fact that parallel transport is an isometry is of interest in the optimization literature, allowing estimates for convergence analysis, see \cite{SatoRiemannianCG,RingWirth}.

\subsection{Outline}In the next section, we review the background required in the paper, including bilinear form and pseudo-Riemannian geometry. We introduce the concept of transposable Lie-subalgebra, which includes $\gl$ and $\oo$. We introduce intertwining operators and the family of deformation metrics in \cref{theo:main} and discuss the condition for the metrics to be Riemmanian.

We show the geodesic and parallel transport formula for general linear groups in \cref{sec:GL} and special orthogonal groups in \cref{sec:SOSE}. In \cref{sec:submerse}, we show the quotient by a subgroup with a particular structure has closed-form geodesics. The parallel transport equation can be reduced to a linear ODE with variable coefficients in the general case, and constant coefficients, solvable by exponential action in special cases. In \cref{sec:stiefel} and \cref{sec:flag}, we provide parallel transport formulas for the Stiefel and flag manifolds, with a discussion of Grassmann manifolds in \cref{app:appGrass}. In \cref{sec:numer}, we provide numerical experiments. We end the paper with a discussion. In \cref{sec:expact} we review the exponential action and prove our norm estimate.

\subsection{Notations} The general linear group $\GL(n)$ is the group of all invertible matrices in $\R^{n\times n}$, the connected subgroup with positive determinant is $\GL^+(n)$, the corresponding Lie algebra $\gl(n)$ could be identified with $\R^{n\times n}$. The special orthogonal group $\SO(n)$ is the group of orthogonal matrices of determinant $1$, its Lie algebra $\oo(n)$ could be identified as a vector space with the space of antisymmetric matrices $\AHerm_n$. Its orthogonal complement is the space $\Herm_n$ of symmetric matrices. The Stiefel, flag, and Grassmann manifolds with notations $\St{n}{p}, \Gr{n}{p}, \Flag(\vd)$ will be introduced later. We write $\sym$ for the symmetrize operator $A_{\sym} = \frac{1}{2}(A+A^{\sfT})$ for $A\in\R^{n\times}$ and $\asym$ for the antisymmetrize operator $A_{\asym} = \frac{1}{2}(A-A^{\sfT})$. The notation $\cP_{_{\cV}}$ denotes the projection to a subspace $\cV$.

In a matrix space $\R^{n\times n}$, we define the $\ad$ operator $\ad_{\ttX}\ttY = [\ttX, \ttY] = \ttX\ttY - \ttY\ttX$ for two matrices $\ttX, \ttY$ and the $\Ad$ operator $\Ad_K\ttX = K\ttX K^{-1}$ for an invertible matrix $K\in \R^{n\times n}$.

\section{Background Review}We treat exponential action in the appendix. The main indefinite linear algebra and differential geometry background are reviewed here.
\subsection{Symmetric Bilinear forms}
As we will discuss pseudo-Riemannian metric, it is appropriate to review some quick facts about bilinear forms, which are indefinite scalar products. While we work with the real case, the proofs and results in the main reference \cite{GohbergIndefinite} for the Hermitian case carry through readily. The main results that we need are about subspaces. A symmetric bilinear form, or just {\it form}, or \emph{scalar product} on a vector space $\cE$ is a scalar-valued bilinear map from $\cE\times\cE$ to $\R$, denoted by $\langle v, w\rangle$. For brevity, we will omit the word {\it symmetric} as we will only work with symmetric bilinear forms. If we identify $\cE$ with $\R^n$ and represent the form by a matrix $\rB$, then $\langle v, w\rangle  = \langle v, w\rangle_{\rB} = v^{\sfT}\rB w$ for $v, w\in\cE$. It is nondegenerate (or we also say $\cE$ is nondegenerate) if $\rB$ is nonsingular, or $\langle v, w\rangle = 0$ for all $w\in\cE$ implies $v=0$. This latter formulation is useful if $\cE$ is a matrix space. We can restrict the form $\langle , \rangle$  to a subspace $\cV\subset \cE$. In contrast to the positive-definite case, an indefinite scalar product may be degenerate on $\cV$ even if it is nondegenerate on $\cE$. We define the orthogonal complement $\cV_{_{\perp}}$ of $\cV$ as follows
\begin{equation}\cV_{_{\perp}} = \{ w\in\cE|\; \langle v, w\rangle = 0\text{ for all } v\in\cV\}.\label{eq:projspace} \end{equation}
Then $\cV_{_{\perp}}$ may intersect $\cV$ non-trivially. The real case of \cite[Proposition 2.2.2]{GohbergIndefinite} generalizes several well-known results for inner product spaces in the lemma below.
\begin{lemma}\label{lem:split} A subspace $\cV\subset \cE$ is nondegenerate if and only if $\cV_{_{\perp}}\cap \cV= \{0\}$. Assuming this condition, we have a unique decomposition $\cE = \cV\oplus \cV_{_{\perp}}$, or any vector $w\in\cE$ could be written in a unique way as $w = w_{\cV} +w_{\cV_{\perp}}$, with $w_{\cV}\in \cV, w_{\cV_{\perp}}\in \cV_{\perp}$.  In that case, the map $\cP_{_{\cV}}: w\mapsto w_{\cV}$ is the orthogonal projection to $\cV$. It satisfies
\begin{equation}  \langle \cP_{_{\cV}}w, v \rangle = \langle w, v\rangle \text{ for } v\in \cV, w\in\cE.\label{lem:orth}\end{equation}
\end{lemma}

\subsection{Manifolds, geodesics, and parallel transport}
For the background on matrix manifolds, the reader can consult \cite{Edelman_1999,Gallier,AMS_book}. We will also need background on \emph{naturally reductive homogeneous spaces}, see \cite{ONeil1983}, reviewed in this section. In the embedded setup, we consider a smooth subset $\cM$ of a matrix space $\cE=\R^{n\times m}$, a {\it manifold}. We equip $\cM$ with a pseudo-Riemannian metric (also called semi-Riemannian metric) $\sfg$, which is a smooth scalar-valued bilinear form $\sfg_X\langle,\rangle$ on the tangent space of $\cM$ at each $X\in\cM$. We write $\sfg\langle,\rangle$ or $\langle,\rangle_{\sfg}$ if the point $X$ is understood. The bilinear form is assumed to be nondegenerate. When it is positive-definite, we have a Riemannian metric, otherwise, the form could have negative but nonzero eigenvalues.

A vector field $\ttY$ on $\cM$ could be considered a $\cE=\R^{n\times m}$-valued function, such that for each $X\in \cM$, $\ttY(X)$ is tangent to $\cM$. If $\ttZ$ is another vector field, we can take the directional derivative $\rD_{\ttY}\ttZ$, considering $\ttZ$ as an $\cE$-valued function. However, $\rD_{\ttY}\ttZ$ is usually not a vector field (the same reason as why an acceleration on a sphere is not tangent), but we can adjust $\rD_{\ttY}\ttZ$ by adding a term of the form $\Gamma(\ttY, \ttZ)$ so that $\nabla_{\ttY}\ttZ := \rD_{\ttY}\ttZ + \Gamma(\ttY, \ttZ)$ is a vector field. The term $\Gamma(\ttY, \ttZ)$ is called a Christoffel function in \cite{Edelman_1999}, while $\nabla_{\ttY}\ttZ$ is called a {\it covariant derivative}, or {\it a connection}. Note that $\Gamma$ is dependent on $X$, and is bilinear in the two tangent vector variables $\xi = \ttY(X), \eta = \ttZ(X)$ for $\xi, \eta\in T_X\cM$. We write $\Gamma(X; \xi, \eta)$ when we need to show the variable $X$ explicitly. For example, for the sphere, we can check that $\Gamma(X; \xi, \eta) = X\xi^{\sfT}\eta$ will make $\rD_{\ttY}\ttZ + \Gamma(\ttY, \ttZ)$ a vector field. 

There is more than one way to define covariant derivatives. However, there is a unique connection satisfying two additional conditions, called the Levi-Civita connection, which will be the kind of connection that we consider in this article. The two conditions are torsion-free, or $\Gamma(\ttY, \ttZ) = \Gamma(\ttZ, \ttY)$ for two vector fields $\ttY, \ttZ$, and metric compatible ($\rD_{\ttY}\langle \ttZ, \ttZ\rangle_{\sfg} = 2\langle \ttZ, \nabla_{\ttY}\ttZ\rangle_{\sfg}$.)

A geodesic from $X$ with initial velocity $\xi\in T_X\cM$ is a curve $\gamma$ on the manifold satisfying the equation $\nabla_{\dot{\gamma}(t)}\dot{\gamma}(t)=0$, with the initial value condition $\gamma(0) = X, \dot{\gamma}(0) = \xi$. The map $v\mapsto \gamma(1)$ is the Riemannian exponential map, which is a different notion than Lie group exponential \cite[section 2.1, chapter 4]{Gallier}.  A central theme that we will exploit is for some special metrics, the two concepts of exponential are related.

If $\eta\in T_X\cM$ is another tangent vector at $X$, then the parallel transport of $\eta$ along the geodesic $\gamma(t)$ from $0$ to $t$ is the vector $\ttT_0^t(\gamma)\eta= \Delta(t)$, where $\Delta$ is a {\it vector field along the curve} $\gamma$, that means for each $0 \leq s\leq t$, $\Delta(s)$ is tangent to $\cM$ at $\gamma(s)$. Additionally, $\Delta$ must satisfy the transport equation 
\begin{equation}\begin{gathered}\dot{\Delta}(s) + \Gamma(\gamma(s); \dot{\gamma}(s), \Delta(s)) = 0,\\
\Delta(0)=\eta.
\end{gathered}
\end{equation}  
We also consider quotient manifolds of matrix manifolds. This is a special case of a pseudo-Riemannian submersion (see \cite[Chapter 7]{ONeil1983}, our main reference for this topic).
In this situation, we assume there is a smooth map $\fq:\bcM\to\cM$ between two pseudo-Riemannian manifolds, such that $\fq$ is surjective and the differential $d\fq$ is surjective at every point. There is a decomposition $T_X\cM=\cV_X\oplus \cH_X$ to the vertical space $\cV_X$, the tangent space to $\fq^{-1}(X)$ and its orthogonal complement $\cH_X$. Vectors in $\cH_X$ are called horizontal vectors, and we can define horizontal vector fields. A tangent vector field on $\cM$ can be lifted uniquely to a horizontal vector field on $\bcM$. Let $\rH$ be the metric-compatible projection to the horizontal space then we have the horizontal lift formula for the Levi-Civita connection \cite[lemma 7.45]{ONeil1983}.
\begin{equation}\rH\overbar{\nabla}_{\bttY}\bttZ = \overline{\left(\nabla_{\ttY}\ttZ \right)}.
\end{equation}  

In the next section, we will define matrix groups, a major topic of this paper. It also provides background for quotient manifolds as examples of submersion.
\subsection{Matrix Lie groups, Lie algebra, left-invariant metrics, and matrix quotient spaces}
A Lie group is a differentiable manifold with a differentiable group structure. We are interested in subgroups of the general linear group $\cE=\GL(n)$, which is the (multiplication) group of invertible matrices in $\R^{n\times n}$. Thus, for such a group $\rG$, $\dI_n\in  \rG$ and $G_1^{-1}G_2\in\rG$ if $G_1, G_2\in\rG$.

We also consider cartesian products of matrix Lie groups, equipping the product $\rG\times \rH$ for two matrix Lie groups $\rG\in \R^{n_1\times n_1}, \rH\in\R^{n_2\times n_2}$ with the component-wise multiplication $(G_1, H_1)(G_2, H_2) = (G_1G_2, H_1H_2)$ and inversion. This is equivalent to embedding $\rG\times\rH$ in a block diagonal matrix subgroup of $\R^{(n_1+n_2)\times (n_1+n_2)}$, where $\rG$ is embedded in the top-left $n_1\times n_1$-block and $\rH$ with the bottom-right $n_2\times n_2$ block. This realization extends to tangent spaces and Lie algebras discussed below.

The tangent space $\fg$ of $\rG$ at $\dI=\dI_n$ is a {\it Lie subalgebra} of $\R^{n\times n}$. Recall a vector subspace $\fg$ of the square matrix space $\cE=\R^{n\times n}$ is called a Lie subalgebra if it is closed under the Lie bracket, that means for two matrices $\ttA, \ttB$ in $\fg$,
\begin{equation}[\ttA, \ttB] :=\ttA\ttB - \ttB\ttA\in \fg.
\end{equation}
The Lie bracket is anticommutative  and satisfies the Jacobi\qs identity \cite[Chapter 3]{Gallier}. The matrix exponential map $\expm$ maps $\fg$ to $\rG$. Conversely, given a Lie subalgebra $\fg$ of $\R^{n\times n}$, we can construct a connected Lie group with Lie algebra $\fg$. In this instant, $\expm$ corresponds to the {\it Lie exponential map} of a subgroup of $\GL(n)$. We have
\begin{equation}\Ad_{\expm(\ttZ)}\ttX = \expm(\ttZ)\ttX \expm(-\ttZ) = \expa(\ad_{\ttZ}, \ttX),\label{eq:expad}
\end{equation}
where $\ad_{\ttZ}: \ttX\mapsto [\ttZ, \ttX]$ and $\expa$ is the exponential action in \cref{sec:expact} \cite[Chapter 3]{Gallier}. From this, $\expm(\ttZ)\ttX\expm(-\ttZ)$ is in $\fg$ for $\ttZ, \ttX\in\fg$. The tangent space of $\rG$ at $X\in\rG$ could be identified with $X\fg$, or $\xi\in T_X\rG$ if and only if $X^{-1}\xi\in\fg$. Thus, if $\langle,\rangle$ is a nondegenerate form on $\fg$, we can define a pseudo-Riemannian metric on $\rG$ by defining for two tangent vectors $\xi, \eta$ at $X$
$\langle \xi, \eta\rangle_X = \langle X^{-1}\xi, X^{-1}\eta\rangle$,
which gives $\rG$ a {\it left-invariant} pseudo-Riemannian metric.

For a left-invariant metric, left multiplication is an isometry, a smooth, invertible map between two manifolds preserving the metric (between the same manifold in this case). A metric is {\it right-invariant} if for all $X_1\in \rG$, the map $R_{X_1}: X\mapsto XX_1, X\in \rG$ is an isometry, and it is \emph{bi-invariant} if it is both left and right invariant, and this is equivalent to $\Ad_{\rG}$-invariant on $\fg$, $\langle K\ttX K^{-1}, K\ttX K^{-1}\rangle =   \langle \ttX , \ttX \rangle$ for $X\in \fg, K\in \rG$. From \cite[Proposition 11.9]{ONeil1983}, if $\rG$ is connected, \emph{bi-invariance} is equivalent to the cyclic invariant condition (note the items in the brackets below preserve the cyclic order)
\begin{equation}\langle \ttA, [\ttB,\ttC] \rangle = \langle \ttB, [\ttC,\ttA] \rangle = \langle \ttC, [\ttA,\ttB] \rangle, \ttA, \ttB, \ttC \in \fg.\label{eq:biinvariant}
\end{equation}
From the same proposition, bi-invariant pseudo-Riemannian metrics on Lie groups have the important property that \emph{geodesics are given by the Lie exponential map}.

The concepts of homogeneous spaces, coset manifolds, group actions (both right and left), and isotropy groups are covered in \cite[Chapters 9 and 11]{ONeil1983}. 

We call a pseudo-Riemannian space $\cM$ a \emph{reductive space} \cite[chapter 11]{ONeil1983} if it is isometric to a quotient $\rG/\rK$, for Lie groups $\rG$ and $\rK$ with Lie algebras $\fg$ and $\fk$ such that there is a subspace $\fm\subset \fg$ complementary to $\fk\subset \fg$ with $\fm$ is $\Ad_{\rK}$-invariant ($K\fm K^{-1}\subset \fm$ for $K\in \rK$, and when $\rK$ is connected, this  is equivalent to $[\fk, \fm] \subset \fm $). From \cite[Proposition 10.3.1]{KobNom}, there is a one-to-one correspondence between $\rG$-invariant pseudo-Riemannian metric on $\rG/\rK$ and $\ad(\fk)$-invariant non-degenerate symmetric bilinear form $\langle, \rangle$ on $\fm$. That means
\begin{equation}\langle [\ttK, \ttM_1], \ttM_2\rangle = \langle \ttM_1, [\ttK, \ttM_2]\rangle  \text{ for } \ttK \in \fk, \ttM_1, \ttM_2\in \fm. \label{eq:ReductiveInv}
\end{equation}
Further, it is \emph{naturally reductive} if the form satisfies the cyclic condition
\begin{equation}\langle \ttM_1, \ttpm[\ttM_2,\ttM_3] \rangle = \langle \ttM_2, \ttpm[\ttM_3,\ttM_1] \rangle \text{ for } \ttM_1, \ttM_2, \ttM_3 \in \fm.\label{eq:NatReductiveInv}
\end{equation}
Here, $\ttpm$ is the projection to $\fm$ (using the fact that $\fk$ and $\fm$ are complementary), the term $\ttpm(\ttM_1)$ is typically written $(\ttM_1)_{\fm}$ in the differential geometric literature, and $\ttpm([\ttM_3,\ttM_1])$ is typically written $[\ttM_1,\ttM_3]_{\fm}$. Here, we follow the matrix convention.

In the above, the form $\langle,\rangle$ needs to be defined on $\fm$ only. We do not assume a metric on $\fk$, or we have much freedom to extend $\langle,\rangle$ to a left-invariant metric on $\rG$, see \cite[Lemma 11.24]{ONeil1983}. However, a construction that is often used is if $\fg$ itself is equipped with a nondegenerate bi-invariant form $\langle,\rangle$ such that its restriction to $\fk$ is nondegenerate, we can take $\fm$ to be the orthogonal complement to $\fk$ in \cref{lem:split}. Then, the restriction to $\fm$ is also nondegenerate, and both \cref{eq:ReductiveInv} and \cref{eq:NatReductiveInv} are satisfied, giving us a naturally reductive metric on $\cM = \rG/\rK$. This is essentially the construction that we use for the nonsingular case in the proof of \cref{theo:main}.


As mentioned in \cite[chapter 11.22]{ONeil1983}, naturally reductive metrics could be considered a generalization of the bi-invariant metrics to a quotient subspace. It has a simple expression for the Levi-Civita connection \cite[Theorem 3.3]{KobNom} and geodesics. In \cite[Proposition 2.12]{SmithThesis}, the author showed if $\rG$ is given a bi-invariant metric (which the proof assumes), the parallel transports on $\rG/\rK$ are also given by a simple equation. This will be the result we use, although we originally derived our results by working directly with the transport equation. 


\section{The Trace form as a bi-invariant symmetric bilinear form and transposable subalgebras}
For matrix Lie algebras, the most important examples of bi-invariant symmetric bilinear forms are trace forms \cite[Lemma 11.6]{ONeil1983}.
The {\it trace form} is the bilinear map on $\R^{n\times n}$, mapping $\ttA, \ttB\in \R^{n\times n}$ to the scalar $\langle \ttA, \ttB\rangle_{\scrK}:=\Tr\ttA\ttB$ (the subscript $\cK$ is in honor of Killing, who defined the closely related Killing form). Using the commutative property of trace, we have 
\begin{equation}\begin{gathered}\Tr (K\ttB K^{-1}) \ttC  =  \Tr\ttB (K^{-1}\ttC K)\quad\text{ for } K\in \GL(n), \ttB, \ttC\in \R^{n\times n},\\
\Tr[\ttA,\ttB]\ttC = \Tr\ttB[\ttC,\ttA]\quad \text{ for } \ttA, \ttB, \ttC\in \R^{n\times n},
\label{eq:cyclicInvariant}
\end{gathered}
\end{equation}
or the trace form is \emph{bi-invariant} on $\R^{n\times n}$ and on all matrix Lie subalgebras. As any strictly upper-triangular matrix $\ttA$ has $\Tr\ttA^2=0$, the trace form is not positive definite. A symmetric matrix $\ttA$ has $\Tr\ttA^2 \geq 0$, while an antisymmetric matrix has $\Tr\ttA^2 \leq 0$. However, if it is nondegenerate on a Lie algebra $\fg$ of a Lie group $\rG$, the trace form induces a bi-invariant pseudo-Riemannian metric on $\rG$ with closed-form geodesics.

The {\it Frobenius form}, or Frobenius inner product $(\ttA, \ttB)\mapsto \langle \ttA, \ttB\rangle_{\scrF} := \Tr\ttA\ttB^{\sfT}$ of two matrices is positive-definite. However, in general, it is not bi-invariant and may not induce a metric with closed-form geodesics. For subspaces that transpose to themselves, the positive-definite property of Frobenius forms can be used to prove nondegeneracy of subspaces equipped with the bi-invariant trace form. 
\begin{definition}A matrix vector space $\cV\subset\R^{n\times n}$ is called transposable if
  \begin{equation}\cV^{\sfT} = \{\ttA^{\sfT} |\ttA\in\cV\} = \cV.\end{equation}
A transposable Lie subalgebra is a transposable vector space that is a Lie subalgebra.
\end{definition}  
The spaces of symmetric, antisymmetric, diagonal matrices, and block-diagonal matrices with fixed block sizes are all transposable. If $\cV$ is transposable then its symmetric $\cV_{\sym}$ and antisymmetric $\cV_{\asym}$ parts are subspaces of $\cV$.

In the following, transposability implies nondegeneracy. This may be well-known in the reductive Lie algebra literature, but we do not have an exact reference.
\begin{proposition} If $\fg$ is transposable then the trace form is nondegenerate on $\fg$. Let $\fv\subset\fg$ be a transposable subspace. 
Let $\fvperp$ be the orthogonal complement of $\fv$ under the trace form $\langle \ttA, \ttB \rangle_{\scrK}= \Tr \ttA\ttB$
\begin{equation}\fvperp = \{ \ttA\in \fg |\; \Tr\ttA\fv = \{0\}\}.
\end{equation}
Then $\fvperp$ is also transposable and is the orthogonal complement under the Frobenius form $\langle \ttA, \ttB \rangle_{\scrF}= \Tr \ttA^{\sfT}\ttB$, and vice-versa. The projection $\cP_{\fv}$ of $\fg$ onto $\fv$ under $\langle,\rangle_{\scrK}$ or $\langle,\rangle_{\scrF}$ is identical and commutes with transpose, $\cP_{\fv}(\ttA^{\sfT}) = \cP_{\fv}(\ttA)^{\sfT}$ for $\ttA\in \fg$.
\end{proposition}
\begin{proof}Note for $\ttA_1, \ttA_2\in \fg$, $\langle \ttA_1, \ttA_2\rangle_{\scrF} = \langle \ttA_1^{\sfT}, \ttA_2\rangle_{\scrK}$, hence, if $\langle \ttA_1, \ttA_2\rangle_{\scrK}=0$ for all $\ttA_2\in \fg$, then $\Tr \ttA_1^{\sfT} \ttA_2=0$ for all $\ttA_2$, therefore $\ttA_1=0$. Thus, the trace form is nondegenerate on $\fg$.

The next statement is a consequence of $\fv=\fv^{\sfT}$, which implies $\fvperp=\{ \ttA\in \fg |\; \Tr\ttA\fv^{\sfT} = \{0\}\}$, or $\fvperp$ is also the orthogonal complement under the Frobenius inner product. Using $\Tr \ttA^{\sfT}\ttB = \Tr \ttA\ttB^{\sfT}$ for two matrices $\ttA, \ttB$, we get $\fvperp$ is transposable. From here, the decomposition $\ttA = \cP_{\fv}\ttA + \cP_{\fvperp}\ttA$ under the trace form is also the decomposition under the Frobenius form. Thus, the projection under one form is the projection under the other. Transposing the decomposition to $\ttA^{\sfT} = (\cP_{\fv}\ttA)^{\sfT} + (\cP_{\fvperp}\ttA)^{\sfT}$ then use the uniqueness of this decomposition, we get $\cP_{\fv}(\ttA^{\sfT}) = (\cP_{\fv}\ttA)^{\sfT}$.
\end{proof}

We now describe projections to several transposable subspaces that we will encounter. A projection to a block diagonal subspace is just setting the entries outside of those blocks to zero. A projection to the space of symmetric matrices is just by applying the symmetrize operator, $A\mapsto A_{\sym}$, projection to the space of antisymmetric matrices is by applying the antisymmetrize operator.  

\section{Pseudo-Riemannian deformation metric}
The main result of this section is if $\rG$ has a non-degenerate bi-invariant pseudo-Riemannian metric $\langle, \rangle$, and $\fa\subset\fg$ is a nondegenerate subalgebra, if we replace the bilinear form on $\fa$ with minus of another bi-invariant form $\langle, \rangle_{\rB}$ in $\fa$ while scaling the form on the complement $\faperp$ by a scalar $\beta_0\neq 0$, then the resulting left-invariant metric on $\rG$ will be naturally reductive (but may no longer be bi-invariant) and we still have closed-form geodesics and tractable parallel transport. Relative to the original form, the form $\langle, \rangle_{\rB}$ on $\fa$ is defined by a symmetric (with respect to $\langle, \rangle$) linear operator $\rB$ on $\fa$ only ($\rB\ttA$ only makes sense if $\ttA\in\fa$), called an intertwining operator. Thus, the new metric will be given by
\begin{equation} \langle \ttA, \ttA\rangle_{_{(\beta_0, \rB)}} := \beta_0\langle \ttpap(\ttA), \ttpap(\ttA)\rangle - \langle \ttpa(\ttA), \rB\ttpa(\ttA)\rangle. \label{eq:metCheegerAlt}
\end{equation}
We first discuss the condition for $\langle,\rangle_{\rB}$ to be bi-invariant.
\subsection{Intertwining operator}
We note that the most important examples of intertwining operators we use are simply multiplying by scalars, so in the procedure above, we just scale the metrics on $\fa$ and $\faperp$ by possibly different scalars. However, the general result requires roughly the same effort to prove, and potentially could have future applications.

If a Lie algebra $\fa$ has a nondegenerate bi-invariant symmetric form $\langle,\rangle$, the form is not unique, as any scalar multiple of $\langle,\rangle$ is also bi-invariant. If $\fa$ is a direct sum of two mutually commuting subalgebras $\fa = \fa_1\oplus \fa_2$, then the form $(\ttA, \ttB)\mapsto \beta_1 \langle \ttA_1, \ttB_1\rangle + \beta_2 \langle \ttA_2, \ttB_2\rangle, \ttA, \ttB\in \fa$ is also bi-invariant, where $\ttA = \ttA_1+\ttA_2, \ttB=\ttB_1+\ttB_2$ is the decomposition of $\ttA$ and $\ttB$ to the $\fa_1$ and $\fa_2$ components. Here, we apply different scalar factors on different components. We note if $\fa$ is commutative, $[ \ttA_1, \ttA_2] = 0$ for all $\ttA_1, \ttA_2\in \fa$ then any form is bi-invariant. These are the often used constructions of bi-invariant forms, see also \cite[Lemma 7.5,7.6]{Milnor1976}, \cite{medinaRevoy,MiolanePennec}.

We now give a condition for an operator to preserve the bi-invariant property. It satisfies the same condition as intertwining operators (morphism of the adjoint representation \cite[section II.6.1]{HumphreysLie}) in representation theory. Thinking of them as generalized scalars would probably be helpful. Most material presented here is likely known, but we could not locate the references.

Note that if $\langle,\rangle$ is non-degenerate, any other symmetric form could be written as $\langle \ttA,  \ttC\rangle_{\rB} := \langle \ttA, \rB \ttC\rangle$ ($\ttA, \ttC\in \fa$) for a symmetric operator $\rB$ (that means 
$\langle \rB \ttA,  \ttC\rangle = \langle \ttA, \rB \ttC\rangle$). Also note that $\rB$ is invertible if and only if $\langle , \rangle_{\rB}$ is non-degenerate.
\begin{definition}
Let $\fa$ be a Lie algebra with a nondegenerate bi-invariant form $\langle,\rangle$, then an intertwining operator $\rB$ is a symmetric operator on $\fa$ such that the form
\begin{equation} (\ttA_1, \ttA_2) \mapsto \langle \ttA_1, \ttA_2\rangle_{\rB} := \langle \ttA_1, \rB\ttA_2\rangle
\end{equation}
is also bi-invariant. 
\end{definition}
We have the following characterization of intertwining operators
\begin{proposition}A symmetric operator $\rB$ is an intertwining operator if and only if $\rB$ satisfies $\rB\ad_{\ttA} = \ad_{\ttA}\rB$ for all $\ttA\in \fa$. Equivalently, it is an intertwining operator if either of the following holds
\begin{gather}
\rB[\ttA_1, \ttA_2] = [\ttA_1, \rB\ttA_2] ,\label{eq:inter1}\\
\rB[\ttA_1, \ttA_2] = [\rB\ttA_1, \ttA_2]. \label{eq:inter2}
\end{gather}
\end{proposition}

\begin{proof}
For $\langle, \rangle_{\rB}$ to be bi-invariant, we need $\langle \ttA_1, [\ttA_2, \ttA_3]\rangle_\rB = \langle \ttA_2, [\ttA_3, \ttA_1]\rangle_\rB $, for $\ttA_1, \ttA_2, \ttA_3\in \fa$ or
$$ \langle \ttA_1, \rB[\ttA_2, \ttA_3]\rangle = \langle \ttA_2, \rB[\ttA_3, \ttA_1]\rangle 
$$
but the first term could be written as $\langle \rB \ttA_1, [\ttA_2, \ttA_3]\rangle =\langle \ttA_2 , [\ttA_3, \rB \ttA_1]\rangle$. This holds for all $\ttA_2$, and since $\langle,\rangle$ is nondegenerate, $[\ttA_3, \rB \ttA_1] = \rB[\ttA_3, \ttA_1]$, or $\rB$ commutes with $\ad_{\ttA_3}$
\begin{equation}\ad_{\ttA}\rB = \rB\ad_{\ttA}\quad \text{ for }\ttA\in \fa.\label{eq:intertwine}
\end{equation}
Finally, note the operator $\ttX\mapsto [\ttX, \ttA]$ is $-\ad_{\ttA}$, thus, if one of \cref{eq:inter1} or \cref{eq:inter2} holds for all $\ttA_1, \ttA_2$ in $\fa$, we have the other.
\end{proof}

In the case of $\oo(n)$, or of a class of Lie algebras called \emph{simple}, $\rB$ must be a scalar multiplication. Most intertwining operators in this paper are of this form. The following shows intertwining operators behave like scalars.
 
\begin{lemma}Let $\dI_{\fa}$ be the identity operator on $\fa$, then $\beta\dI_{\fa}, \beta\in \R$, is an intertwining operator. Sum, product, and inversion (if exists) of intertwining operators are intertwining operators.\label{lem:interAlg}
\end{lemma}
For example, if $\rB_1$ and $\rB_2$ are intertwining operators then
$$\ad_{\ttA}\rB_1\rB_2 = \rB_1\ad_{\ttA}\rB_2 = \rB_1\rB_2\ad_{\ttA},
$$
and $\ad_{\ttA}\rB(\rB^{-1}\ttA_1) = \rB\ad_{\ttA}(\rB^{-1}\ttA_1)$ implies $\rB^{-1}\ad_{\ttA}\ttA_1 = \ad_{\ttA}(\rB^{-1}\ttA_1)$, proving product and inversion producing new intertwining operators.
\subsection{Deformation using an intertwining operator on a subalgebra}
We will consider matrix groups in the following theorem to simplify the notation, but we believe it also holds for abstract Lie groups, replacing matrix multiplications with left and right multiplication actions.

\begin{proposition}\label{prop:main}
Let $\rG$ be a connected matrix Lie group with Lie algebra $\fg$ and a nondegenerate bi-invariant form $\langle,\rangle$. Let $\rA$ be a connected Lie subgroup of $\rG$ with Lie algebra $\fa \subset \fg$ such that $\langle,\rangle$ restricted to $\fa$ is nondegenerate. Let $\rB$ be an invertible intertwining operator on $\fa$. Write $\vbeta$ for the tuple $(\beta_0, \rB)$ where $\beta_0 \neq 0$ is a real number. For $\ttX_1, \ttX_2\in \fg$, define another bilinear form on $\fg$
\begin{equation}\langle \ttX_1, \ttX_2\rangle_{\vbeta} := \beta_0\langle \ttX_1, \ttX_2\rangle - \langle (\rB+ \beta_0\dI_{\fa})\ttpa\ttX_1, \ttpa\ttX_2\rangle.
\label{eq:metCheeger}
\end{equation}
It induces a left-invariant metric on $\rG$. Let $\trG $ be the component-wise product $\rG\times \rA$ and $\trH = \{ (H, H)\in \trG| H \in \rA\}$. 
Then the action $(X, A).Y= XYA^{-1}$ makes $\trG$ a group of isometries of $\rG$ with the metric \cref{eq:metCheeger}. The isotropy group at $\dI_{\rG}$ is $\trH$. Thus, we can identify $\rG$ with $\trG/\trH$. 
\end{proposition}
\begin{proof}First, \cref{eq:metCheeger} is equivalent to \cref{eq:metCheegerAlt} since $\cP_{\faperp} = \dI_{\fg} - \cP_{\fa}$, and $\langle \ttpa \ttX_1, \ttX_2 \rangle = \langle \ttpa \ttX_1, \ttpa\ttX_2 \rangle$. The form $\langle,\rangle$ is nondegenerate on $\faperp\subset \fg$ by \cref{lem:split}. If $\ttX_1$ is a degenerate element in \cref{eq:metCheegerAlt}, it pairs to zero with all $\ttX_2\in \fg$. For $\ttX_2$ in $\fa$, this implies $\ttpa\ttX_1=0$, while for $\ttX_2 \in \faperp$, we get $\cP_{\faperp}\ttX_1 = 0$. Together, we have $\ttX_1 =0$. Thus, the form \cref{eq:metCheeger} is nondegenerate.

By definition, the metric \cref{eq:metCheeger} is left invariant. To show for $A\in \rA$, $Y\mapsto YA^{-1}$ is an isometry on $\rG$, the first term of \cref{eq:metCheeger} is bi-invariance. We need to deal with the second term.
We have $\Ad_A$ commutes with $\ttpa$ since for $\ttX_1, \ttX_2\in \fg$, 
$$\begin{gathered}\langle \ttpa \Ad_A \ttX_1, \ttX_2  \rangle = \langle \Ad_A \ttX_1 , \ttpa \ttX_2  \rangle  = \langle  \ttX_1 , \Ad_{A^{-1}}\ttpa \ttX_2 \rangle 
= \langle  \ttpa\ttX_1 , \Ad_{A^{-1}}\ttpa \ttX_2 \rangle \\
= \langle \Ad_A \ttpa\ttX_1 , \ttpa \ttX_2  \rangle = \langle \Ad_A \ttpa\ttX_1 , \ttX_2  \rangle.
\end{gathered}$$
Thus, $\ttpa \Ad_A \ttX_1 = \Ad_A \ttpa\ttX_1$. Since $\rA$ is connected, by \cref{eq:expad}, $\rB$ commutes with $\Ad$. Hence, for $\xi, \eta\in T_Y\rG$, we have $\xi A^{-1}, \eta A^{-1}\in T_{YA^{-1}}\rG$ and we need to evaluate
$$\begin{gathered}\langle (\rB+ \beta_0\dI_{\fa})\ttpa\left((YA^{-1})^{-1}\xi A^{-1}\right), \ttpa\left((YA^{-1})^{-1}\eta A^{-1}\right)\rangle\\
= \langle (\rB+ \beta_0\dI_{\fa})\ttpa \left( A(Y^{-1}\xi) A^{-1}\right), \ttpa \left(A(Y^{-1}\eta) A^{-1}\right)\rangle \\
= \langle A(\rB+ \beta_0\dI_{\fa})\ttpa (Y^{-1}\xi) A^{-1}, A\ttpa (Y^{-1}\eta) A^{-1}\rangle
= \langle (\rB+ \beta_0\dI_{\fa})\ttpa (Y^{-1}\xi) , \ttpa (Y^{-1}\eta) \rangle
\end{gathered}$$
by $\Ad_A$ invariance. Thus, $\trG$ acts as an isometry. The isotropy group at $\dI_{\rG}$ consists of $(X, A)$ such that $XA^{-1}=\dI_{\rG}$, thus $X=A$, proving the isotropy group is $\trH$. Observe for $X\in \rG$ and $A\in \rA$, the action of $\trG$ by $(XA, A)$ sends $\dI_{\rG}$ to $XAA^{-1}=X$, thus the action is transitive. The inverse image at each $X$ is a copy of $\rA$, embedded as $\tilde{\rH}$. This identifies $\rG$ with $\trG/\trH$ as manifolds.
\end{proof}
\begin{theorem}\label{theo:main}The metric in \cref{prop:main} is naturally reductive when $\rG$ is identified with $\trG/\trH$. Specifically, let $\tfg = \fg\times \fa$ be the Lie algebra of $\trG$ and $\tfh$ be the Lie algebra of $\trH$. Let $\hrB := (\rB+\beta_0\dI_{\fa})\rB^{-1}$. Then
\begin{equation}\begin{gathered}
\tfh = \{ (\ttA, \ttA)\in \tfg| \ttA \in \fa\},\\
\tfm := \{(\ttX, \hrB\ttpa\ttX)|\ttX\in\fg \} \text{ with }\hrB = (\rB+\beta_0\dI_{\fa})\rB^{-1}
\label{eq:tfhm}
\end{gathered}
\end{equation}
are complementary subspaces, $\tfg = \tfm+\tfh$ and $\tfm\cap \tfh = \{0\}$ with the decomposition
\begin{equation}\begin{gathered}(\ttX, \ttA) = \left(\ttX
- (\dI_{\fa}+\beta_0^{-1}\rB)\ttpa\ttX + \beta_0^{-1}\rB\ttA
, -(\dI_{\fa}+\beta_0^{-1}\rB)\ttpa(\ttX-\ttA)\right)\\
+ \left(\ttA + (\dI_{\fa}+\beta_0^{-1}\rB)\ttpa(\ttX-\ttA) , \ttA + (\dI_{\fa}+\beta_0^{-1}\rB)\ttpa(\ttX-\ttA)\right).\label{eq:decompose}
\end{gathered}\end{equation}
The space $\tfm$ is $\Ad_{\trH}$ invariance. The following bilinear form on $\tfm$
\begin{equation}\langle (\ttX_1, \hrB\ttpa\ttX_1) , (\ttX_2, \hrB\ttpa\ttX_2)\rangle_{\tfm} := \beta_0\langle \ttX_1, \ttX_2 \rangle -\langle \beta_0^2\rB^{-1}(\dI_{\fa} +\beta_0^{-1}\rB) \ttpa\ttX_1, \ttpa\ttX_2\rangle
\label{eq:metricM}
\end{equation}
satisfies the naturally reductive condition, inducing the metric \cref{eq:metCheeger} on $\rG$. 

If $\dI_{\fa} +\beta_0^{-1}\rB$ is invertible then for $(\ttX_1, \ttA_1), (\ttX_2, \ttA_2)\in \fg\times\fa$, define the metric
\begin{equation}\langle (\ttX_1, \ttA_1), (\ttX_2, \ttA_2)\rangle_{\trG,\vbeta} := \beta_0\langle \ttX_1, \ttX_2\rangle - \langle \beta_0\rB(\rB +\beta_0\dI_{\fa})^{-1}\ttA_1, \ttA_2\rangle. \label{eq:metGH}
\end{equation}
It is bi-invariant on $\trG$ and restricts to \cref{eq:metricM}. The mapping $\fq:\trG\to\rG$ defined by $\fq: (X, A)\mapsto XA^{-1}\in\rG$ is a pseudo-Riemannian submersion between $\langle, \rangle_{\trG,\vbeta}$ and $\langle, \rangle_{\vbeta}$.
\end{theorem}
\begin{remark}
    1. Since $\rB$ is symmetric and $\ttpa$ is a projection, in \cref{eq:metCheeger}
    \begin{equation} \langle (\rB+ \beta_0\dI_{\fa})\ttpa\ttX_1, \ttpa\ttX_2\rangle =  \langle \ttpa\ttX_1, (\rB+ \beta_0\dI_{\fa})\ttpa\ttX_2\rangle 
=  \langle \ttX_1, (\rB+ \beta_0\dI_{\fa})\ttpa\ttX_2\rangle.\label{eq:Bbetaeq}
\end{equation}
    2. Note, when $\rB= \beta\beta_0\dI_{\fa}$, singularity of $\dI_{\fa} +\beta_0^{-1}\rB$ means $\beta=-\beta_0$, thus, the metric \cref{eq:metCheeger} is a multiple of the original bi-invariant metric. In this case, while nominally we deform using $\fa$, there is no actual deformation. For general $\rB$, with appropriate conditions, we can likely show if $\dI_{\fa} +\beta_0^{-1}\rB$  is singular, deformation is practically done on a subalgebra $\fa_1\subset \fa$. We will not use this unproven statement.\hfill\break
    3. The decomposition \cref{eq:decompose} and the form \cref{eq:metricM} on $\tfm$ are sufficient to define a metric on $\rG$ identified with $\trG/\trH$, the naturally reductive conditions 
    \cref{eq:ReductiveInv} and \cref{eq:NatReductiveInv} can be verified directly without assuming a metric on $\trG$. In actuality, we arrive at the decomposition and the metric on $\tfm$ assuming
    $\dI_{\fa} +\beta_0^{-1}\rB$ is nonsingular, following existing literature on the deformation metric \cite{Cheeger1973,Jensen,VMSil,ExtCurveStiefel,DAZ}, then observe that the decomposition and the metric on $\tfm$ are both well-defined even
    when $\dI_{\fa} +\beta_0^{-1}\rB$ is singular. We will present a proof of the theorem along this line, using the bi-invariance metric \cref{eq:metGH} then show the general case follows from a continuity argument.
\end{remark}

\begin{proof}From \cref{eq:tfhm}, if  $(\ttX, \ttA) \in \tfh\cap \tfm$ then  $\ttX=\ttA$ and $\hrB\ttA =\ttA$, but this shows $\beta_0\rB^{-1}\ttA=0$, thus $\ttX=\ttA=0$, or $\tfh\cap \tfm=\{0 \}$.

In \cref{eq:decompose}, it is clear we have the equality and that the second line is in $\tfh$. To check that the first term is in $\tfm$ so that \cref{eq:decompose} is a direct sum decomposition, we expand
$$\hrB(\ttpa\ttX
- (\dI_{\fa}+\beta_0^{-1}\rB)\ttpa\ttX + \beta_0^{-1}\rB\ttA) = (\rB+\beta_0\dI_{\fa})\rB^{-1}(
- (\beta_0^{-1}\rB)\ttpa\ttX + \beta_0^{-1}\rB\ttA),$$
which simplifies to $-(\dI_{\fa}+\beta_0^{-1}\rB)\ttpa(\ttX-\ttA)$, the $\rA$-component of the first line.

Applying partial derivative rule on $XA^{-1}$, the differential $d\fq_{(X,A)}$ of $\fq$ applied on direction $(\eta^{\rG}, \eta^{\rA})\in T_{(X, A)}\trG$ is given by
\begin{equation}d\fq_{(X,A)}: (\eta^{\rG}, \eta^{\rA})\mapsto X(X^{-1} \eta^{\rG} -  A^{-1}\eta^{\rA})A^{-1} = XA^{-1}\Ad_A(X^{-1} \eta^{\rG} -  A^{-1}\eta^{\rA}).\label{eq:dqCheeger}\end{equation}
In particular, at $\dI_{\trG}$, it maps $(\eta^{\rG}, \eta^{\rA})$ to $\eta^{\rG} -\eta^{\rA}$. Thus, for $\ttX\in \fg$, $d\fq_{_{\dI_{\trG}}}:(\ttX, \hrB\ttpa\ttX) \mapsto \ttX - (\rB + \beta_0\dI_{\fa} )\rB^{-1}\ttpa\ttX\in \fg$. Considered as a map from $\tfm$ to $\fg$, its inverse applied on $\ttZ \in\fg$ is the \emph{horizontal lift} of $\ttZ$, given by the element $\tilde{\ttZ}\in \tilde{\fm}$ below
\begin{equation}\tilde{\ttZ} =
(\ttZ - (\dI_{\fa}+\beta_0^{-1}\rB)\ttpa\ttZ, -(\dI_{\fa}+\beta_0^{-1}\rB)\ttpa\ttZ) .
\label{eq:horLiftIntert}\end{equation}
To see this, it is immediate that $d\fq_{_{\dI_{\trG}}}\tilde{\ttZ} =\ttZ$. Let $(\ttZ^{\rG}, \ttZ^{\rA})$ be the components of $\tilde{\ttZ}$
\begin{equation}\tilde{\ttZ} = (\ttZ^{\rG}, \ttZ^{\rA}) = (\ttZ + \ttZ^{\rA}, \ttZ^{\rA});\quad \ttZ^{\rA} = -(\dI_{\fa}+\beta_0^{-1}\rB)\ttpa\ttZ.\label{eq:xiGA}
\end{equation}
We show $\tilde{\ttZ}$ is in $\tfm$ by verifying $\hrB\ttpa(\ttZ + \ttZ^{\rA}) = \ttZ^{\rA}$ by simple algebra
$$(\rB+\beta_0\dI_{\fa})\rB^{-1} \ttpa(\ttZ -(\dI_{\fa}+\beta_0^{-1}\rB)\ttpa\ttZ)=(\rB+\beta_0\dI_{\fa})\rB^{-1} ( -(\beta_0^{-1}\rB)\ttpa\ttZ)=\ttZ^{\rA}.
$$
A direct proof of natural reductivity is to expand the $\tfm$-form \cref{eq:metricM} for $\tilde{\ttZ}$ and verify it simplifies to \cref{eq:metCheeger}, and similarly verifying $\Ad_{\trH}$-invariance and \cref{eq:NatReductiveInv}  directly. However, as discussed, we give a more comprehensive proof using the metric \cref{eq:metGH} below.

When $\dI_{\fa} +\beta_0^{-1}\rB$ is nonsingular, the metric \cref{eq:metGH} is bi-invariant on $\trG$, as the $\rG$-component is invariant by the bi-invariance of $\fg$, while the $\rA$-component is bi-invariant by \cref{lem:interAlg} (for example, it reduces to $(\ttA_1, \ttA_2)\mapsto \frac{\beta}{\beta + 1}\langle \ttA_1, \ttA_2\rangle$ if $\rB = \beta\beta_0\dI_{\fa}$). When $\dI_{\fa} +\beta_0^{-1}\rB$ is singular, the metric is not defined on $\trG$. However, the form \cref{eq:metricM} on $\tfm$ is still well-defined and continuous in the parameters. Therefore, if the naturally reductive conditions on $\tfm$ hold for nonsingular $\dI_{\fa} +\beta_0^{-1}\rB$, the conditions carry to singular values by continuity. Hence, it suffices to prove the nonsingular case, where natural reductivity follows from the bi-invariance of \cref{eq:metGH}, since both conditions \cref{eq:ReductiveInv} and \cref{eq:NatReductiveInv} are instances of the bi-invariance condition.

First, we need to show the decomposition \cref{eq:decompose} is the orthogonal decomposition under \cref{eq:metGH}. The orthogonal complement of $\tfh$ under the metric \cref{eq:metGH} consists of pairs $(\ttX, \ttA)\in \fg\times\fa$ satisfying $$\beta_0\langle \ttX, \ttA_0\rangle - \langle \beta_0\rB(\rB +\beta_0\dI_{\fa})^{-1}\ttA, \ttA_0\rangle =0
$$
for all $\ttA_0\in\fa$. Since $\langle \ttX, \ttA_0\rangle = \langle \ttpa\ttX, \ttA_0\rangle$, the above reduces to $\ttpa\ttX-\rB(\rB +\beta_0\dI_{\fa})^{-1}\ttA =0$, 
or $\ttpa\ttX=\rB(\rB +\beta_0\dI_{\fa})^{-1}\ttA$, thus $\ttA = \hrB\ttpa \ttX$. Hence, the orthogonal complement of $\tfh$ is $\tfm$, and the metric \cref{eq:metGH} is nondegenerate on $\tfh$ and $\tfm$ by \cref{lem:split}.

We now verify the metric \cref{eq:metricM} is the restriction of \cref{eq:metGH} to $\tfm$
$$\begin{gathered}\langle (\ttX, \hrB\ttpa\ttX), (\ttX, \hrB\ttpa\ttX) \rangle_{\trG, \vbeta} = \beta_0\langle \ttX, \ttX\rangle - \langle
\beta_0\rB(\rB +\beta_0\dI_{\fa})^{-1}\hrB\ttpa\ttX , \hrB\ttpa\ttX
\rangle\\
=\beta_0\langle \ttX, \ttX\rangle - \langle
\beta_0\rB(\rB +\beta_0\dI_{\fa})^{-1} (\rB+\beta_0\dI_{\fa})\rB^{-1}\ttpa\ttX ,  (\rB+\beta_0\dI_{\fa})\rB^{-1}\ttpa\ttX
\rangle\\
=\beta_0\langle \ttX, \ttX\rangle - \langle
\beta_0\ttpa\ttX ,  (\rB+\beta_0\dI_{\fa})\rB^{-1}\ttpa\ttX
\rangle,
\end{gathered}$$
which is \cref{eq:metricM}. To show $d\fq$ is a submersion, let $\tilde{\ttX}\in \tilde{\fm}$ be the lift \cref{eq:horLiftIntert} of $\ttX$, then
$$\begin{gathered}\langle \tilde{\ttX}, \tilde{\ttX}\rangle_{\trG, \vbeta} = \beta_0\langle \ttX - (\dI_{\fa} + \beta_0^{-1}\rB)\ttpa\ttX, \ttX - (\dI_{\fa} +\beta_0^{-1}\rB)\ttpa\ttX\rangle \\
- \langle \beta_0\rB(\rB +\beta_0\dI_{\fa})^{-1}(-(\dI_{\fa} + \beta_0^{-1}\rB)\ttpa\ttX), -(\dI_{\fa}+\beta_0^{-1}\rB)\ttpa\ttX\rangle\\
= \beta_0\langle \ttX, \ttX \rangle
-2\beta_0\langle \ttpa\ttX, (\dI_{\fa} +\beta_0^{-1}\rB)\ttpa\ttX \rangle + \beta_0\langle (\dI_{\fa} +\beta_0^{-1}\rB)\ttpa\ttX, (\dI_{\fa} +\beta_0^{-1}\rB)\ttpa\ttX\rangle\\
- \langle \rB\ttpa\ttX, (\dI_{\fa}+\beta_0^{-1}\rB)\ttpa\ttX\rangle\\
= \beta_0\langle \ttX, \ttX \rangle -\langle 2\beta_0\ttpa\ttX - \beta_0(\dI_{\fa} +\beta_0^{-1}\rB)\ttpa\ttX + \rB\ttpa\ttX, (\dI_{\fa} +\beta_0^{-1}\rB)\ttpa\ttX \rangle \\
=\beta_0\langle \ttX , \ttX \rangle -\beta_0 \langle \ttpa\ttX,  (\dI_{\fa}+\beta_0^{-1}\rB)\ttpa\ttX\rangle, 
\end{gathered}$$
where between the third to last and the second to last expression, we factor out $\langle , (\rB+\beta_0\dI_{\fa})\rB^{-1}\ttpa\ttX \rangle$, then simplify to get to the last expression, 
which is the metric in \cref{eq:metCheeger}. We have previously showed $d\fq_{_{\dI_{\trG}}}$ maps $\tfm$ bijectively onto $\fg$ with the inverse given by the lift \cref{eq:horLiftIntert}. In particular, $d\fq_{_{\dI_{\trG}}}$ is surjective, and the invariance of the metric shows it is surjective at any point in $\trG$. Thus, $\fq$ is a submersion. Since \cref{eq:metricM} corresponds to the nondegenerate form \cref{eq:metCheeger} except for a singular pair $(\beta_0, \rB)$, it corresponds everywhere by continuity. All the required identities hold, hence \cref{eq:metCheeger} is naturally reductive.
\end{proof}
\begin{theorem}\label{theo:mainChris}
With the metric \cref{eq:metCheeger}, the Christoffel function of two tangent vectors $\xi, \eta$ at $X\in \rG$ is given by
\begin{equation}\begin{gathered}
\Gamma(\xi, \eta) =  -\frac{1}{2}(\xi X^{-1}\eta+\eta X^{-1}\xi) \\
    + \frac{1}{2}X\left([(\dI_{\fa}+\beta_0^{-1}\rB)\ttpa(X^{-1}\xi),X^{-1}\eta] + [(\dI_{\fa}+\beta_0^{-1}\rB)\ttpa(X^{-1}\eta),X^{-1}\xi] \right).
\end{gathered}\label{eq:ChristoffelCheeger}
\end{equation}
The geodesic starting at $X\in\rG$ with initial velocity $\xi\in T_X\rG$ is given by
\begin{equation}\gamma(t) = X\expm(t\left(-(\dI_{\fa}+\beta_0^{-1}\rB)\ttpa(X^{-1}\xi) + X^{-1}\xi\right))\expm(t(\dI_{\fa}+\beta_0^{-1}\rB)\ttpa(X^{-1}\xi)  ).\label{eq:geodesicCheeger}\end{equation}
  Set $\ttV = X^{-1}\xi$, let $\eta\in T_X\cM$ be another tangent vector, then the parallel transport of $\eta$ along the geodesics $\gamma$ to $X_t=\gamma(t)$ with initial point and velocity $(X, \xi)$ is
  \begin{equation}\begin{gathered}
    \ttT(\gamma)_{0}^t\eta = X\expm(t\left(-(\dI_{\fa}+\beta_0^{-1}\rB)\ttpa\ttV + \ttV\right))\expa(t\rP_{\ttV}, X^{-1}\eta)\expm(t(\dI_{\fa}+\beta_0^{-1}\rB)\ttpa\ttV  )\\
    \text{ where for }\ttB\in \fg, \rP_{\ttV}: \ttB \mapsto \frac{1}{2}\left( [\ttB, \ttV] + [(\dI_{\fa}+\beta_0^{-1}\rB)\ttpa\ttV, \ttB]-[(\dI_{\fa}+\beta_0^{-1}\rB)\ttpa\ttB, \ttV]\right).\label{eq:transportCheeger}
\end{gathered}    
\end{equation}
The map $\rP_{\ttV}$ is antisymmetric in $\langle, \rangle_{\vbeta}$, $\langle \rP_{\ttV}\ttB, \ttC\rangle_{\vbeta} + \langle \rP_{\ttV}\ttC, \ttB\rangle_{\vbeta} =0$ for $\ttV, \ttB, \ttC\in \fg$.
\end{theorem}
\begin{proof}We will assume $\dI_{\fa}+\beta_0^{-1}\rB$ is invertible, then use continuity as before.

We start with the geodesics. For the geodesic at point $X\in \rG$ with initial velocity $\xi$, we lift $X$ to $(X, \dI_{\rA})$. Then the geodesic $(X(t), A(t))$ on $\trG$ using the horizontal lift of $X^{-1}\xi$ in \cref{eq:horLiftIntert} is given by
$$\left(X\expm\left(t(X^{-1}\xi - (\dI_{\fa} + \beta_0^{-1}\rB)\ttpa(X^{-1}\xi))\right), \expm \left(t\left(-\beta_0^{-1}(\rB+\beta_0\dI_{\fa})\ttpa(X^{-1}\xi)\right)\right)\right)
$$
which maps by the submersion $\fq$ to the geodesic \cref{eq:geodesicCheeger}.

We will compute the connection at $\dI_{\rG}$ then use the invariance property. For the geodesic starting at $X(0) = \dI_{\rG}$ with initial velocity $\dot{X}(0) = \xi$ of the form $X(t) = \expm(t(-M_1+M_2))\expm(tM_1)$ for $M_1, M_2\in \fg$, we have $\dot{X}(t) = \expm(t(-M_1+M_2)(M_2)\expm(tM_1)$ and 
$$\ddot{X}(0) =- \Gamma(\xi, \xi)= (-M_1+M_2)M_2+M_2M_1 = -[M_1, M_2]+M_2^2.$$
We thus get $\Gamma(\xi, \xi) = -\xi^2 + [(\dI_{\fa}+\beta_0^{-1}\rB)\ttpa(\xi), \xi]$ and \cref{eq:ChristoffelCheeger} follows by polarization and left invariance.

For parallel transport, we also start with $X=\dI_{\rG}$ then use left invariance. From \cite[Proposition 2.12]{SmithThesis}, the lifted parallel transport from $(\dI_{\rG}, \dI_{\rA})$ is $
(\eta(t)^{\rG},\eta(t)^{\rA}) = (\expm(t(\xi+\xi^{\rA}))\ttC^{\rG}, \expm(t\xi^{\rA})\ttC^{\rA})$ where $(\ttC^{\rG}(t), \ttC^{\rA}(t))\in \fg\times \fa$ solves
$$\begin{gathered}(\dot{\ttC}^{\rG}, \dot{\ttC}^{\rA} ) = \cP_{\tilde{\fm}}(-\frac{1}{2}[\xi + \xi^{\rA},  \ttC^{\rG}], -\frac{1}{2}[\xi^{\rA}, \ttC^{\rA}]),\\
(\ttC^{\rG}(0), \ttC^{\rA}(0)) = (\eta + \eta^{\rA}, \eta^{\rA}).
\end{gathered}$$
Since both the initial condition and the equation are in $\tilde{\fm}$, $(\ttC^{\rG}, \ttC^{\rA})\in\tilde{\fm}$. Using \cref{eq:dqCheeger}, the lifted transport maps to the transport on $\rG$ as
$$X(t)(X(t)^{-1}\eta(t)^{\rG} -A(t)^{-1}\eta(t)^{\rA} )A(t)^{-1} = \expm(t(\xi + \xi^{\rA}))(\ttC^{\rG} - \ttC^{\rA} ) \expm( -t\xi^{\rA}).
$$
Let $F(t) = \ttC^{\rG}(t) - \ttC^{\rA}(t)$, thus, $(\ttC^{\rG}(t), \ttC^{\rA}(t))$ is the lift of $F(t)$ by $d\fq_{(\dI_{\rG}, \dI_{\rA})}$. We have $F(0) =\eta$, and by \cref{eq:horLiftIntert}, $\rC^{\rG}= F+\rC^{\rA} , \rC^{\rA} = - (\dI_{\fa}+ \beta^{-1}_0\rB)\ttpa F$. The ODE for $F$ is
$$\begin{gathered}\dot{F} = \dot{\ttC}^{\rG}- \dot{\ttC}^{\rA} = -\frac{1}{2}\left([\xi+\xi^{\rA}, F+\rC^{\rA}] - [\xi^{\rA}, \rC^{\rA}]  \right) \\
= -\frac{1}{2}\left([\xi, F] - [\xi, (\dI_{\fa} + \beta_0^{-1}\rB)\ttpa F] - [(\dI_{\fa} + \beta_0^{-1}\rB)\ttpa\xi, F]\right).
\end{gathered}$$
Thus, $F(t)=\expa(t\rP_{\xi}, \eta)$, giving us the parallel transport formula \cref{eq:transportCheeger} for $X=\dI_{\rG}$. To show $\rP_{\ttV}$ (where $\ttV = \xi$) is antisymmetric in $\langle,\rangle_{\vbeta}$, set $R=\dI_{\fa} + \beta_0^{-1}\rB$, we calculate
$$\begin{gathered}\frac{2}{\beta_0}\langle\rP_{\ttV}\ttB, \ttC \rangle_{\vbeta} =  \langle[\ttB, \ttV]+ [R\ttpa\ttV, \ttB]- [R\ttpa\ttB, \ttV],\ttC\rangle    - \langle [\ttB, \ttV] + [R\ttpa\ttV, \ttB]-[R\ttpa\ttB, \ttV], R\ttpa \ttC\rangle
    \end{gathered}$$
from \cref{eq:Bbetaeq}. Expand the right hand side and transform $\langle [R\ttpa\ttV, \ttB], R\ttpa\ttC\rangle$ to
$$\begin{gathered}\langle \ttB, [R\ttpa\ttC,R\ttpa\ttV ]\rangle = \langle \ttpa\ttB, [R\ttpa\ttC,R\ttpa\ttV ]\rangle 
= \langle \ttpa\ttB, R[R\ttpa\ttC,\ttpa\ttV ]\rangle  \\
= \langle R\ttpa\ttB, [R\ttpa\ttC, \ttpa\ttV]\rangle 
= \langle \ttpa\ttV,  [R\ttpa\ttB, R\ttpa\ttC]\rangle 
=\langle \ttV,  [R\ttpa\ttB, R\ttpa\ttC]\rangle.
\end{gathered}$$
Rearrange the other terms in $\frac{2}{\beta_0}\langle\rP_{\ttV}\ttB, \ttC \rangle_{\vbeta}$ to isolate $\ttV$ to one side using bi-invariance, 
$$\begin{gathered}
\frac{2}{\beta_0}\langle\rP_{\ttV}\ttB, \ttC \rangle_{\vbeta} =
  \langle \ttV, [\ttC, \ttB]\rangle + \langle R\ttpa\ttV, [\ttB, \ttC]\rangle-\langle \ttV, [\ttC, R\ttpa\ttB]\rangle
- \langle \ttV, [R\ttpa\ttC, \ttB]\rangle \\
- 2\langle\ttV, [R\ttpa\ttB, R\ttpa\ttC ] \rangle
\end{gathered}$$
Thus, by anticommutativity of the Lie bracket, we have $\langle\rP_{\ttV}\ttB, \ttC \rangle_{\vbeta} + \langle\rP_{\ttV}\ttC, \ttB \rangle_{\vbeta} = 0$.
\end{proof}

\subsection{Riemannian metrics from the family}\label{sec:riemannFam}We will assume $\langle, \rangle$ is given by the trace form, and the algebras are transposable. The case $\rB=-\beta_0\dI_{\alpha}$ means the metric \cref{eq:metCheeger} is bi-invariant, a multiple $\beta_0\Tr \ttA_1\ttA_2$ of the trace form. In \cref{theo:mainChris}, the connection $\Gamma(\xi, \eta)$ reduces to the first two terms, the geodesic is $X\expm(tX^{-1}\xi)$, or just $\expm(t\xi)$ if $X=\dI_{\cE}$, the operator $\rP_{\ttV}$ is just $-\frac{1}{2}\ad_{\ttV}$, and using \cref{eq:expad}, the parallel transport is $X\expm(\frac{t}{2}\ttV)X^{-1}\eta\expm(\frac{t}{2}\ttV)$. This is known \cite{SmithThesis}.

However, the trace form metric may not be Riemannian. Using \cref{theo:main} to deform the trace form using a Lie algebra $\fa$, for certain choices of the subgroup $\fa$ and parameters $\vbeta$, the metric becomes Riemannian. Let $\sfT$ be the transpose operator, $\sfT: A\mapsto A^{\sfT}$. Compared with the Frobenius form, the deformation metric becomes
\begin{gather}\langle \ttA, \ttB\rangle_{\vbeta} = \langle \ttA, \cI_{\vbeta}\ttB\rangle_{\scrF},\\
\cI_{\vbeta}:= \sfT\circ \left(\beta_0(\rI_{\cE} - \ttpa) - \rB\circ\ttpa\right)
= \left(\beta_0(\rI_{\cE} - \ttpa) -\rB\circ\ttpa\right)\circ \sfT.
\end{gather}
In the above, we identify an element in the range of $\rB$ with an element of $\fg$.

We will discuss a more general example in \cref{sec:GL}, but will focus on the case $\rB = \beta_1\dI_{\fa}$  otherwise. The transpose $\sfT$ commutes with $\ttpa$ since we consider transposable algebras. We will study the eigenvalues of $\cI_{\vbeta} = \sfT\circ \left(\beta_0(\rI_{\cE} - \ttpa) -\beta_1\ttpa\right)$.

Note that the transpose $\sfT$ has eigenvalues $\pm 1$ on the spaces of symmetric and antisymmetric matrices. Write $\cV_{\sym}$ and $\cV_{\asym}$ for the symmetric and antisymmetric components of a transposable subspace $\cV\subset \cE$, the possible eigenpairs of $\cI_{\vbeta}$ are $(\beta_0, (\faperp)_{\sym}), (-\beta_0, (\faperp)_{\asym}), (-\beta_1, \fa_{\sym}), (\beta_1, (\fa)_{\asym})$. The positive-definite requirement of a Riemannian metric implies at least one of the first two, and at least one of the last two eigenspaces are trivial, assuming the appropriate signs of $\beta_0$ and $\beta_1$. These are the possible cases
\begin{enumerate}\item $(\faperp)_{\sym}=\{0\}, (\fa)_{\sym}=\{0\}$. Thus, $\fg_{\sym}=\{0\}$, or $\rG$ is a compact Lie subgroup of $\SO(n)$. In this case, $\beta_0 <0, \beta_1 > 0$. An example is $\SO(n)$ itself in \cref{sec:SOSE}.
\item  $(\faperp)_{\sym}=\{0\}, (\fa)_{\asym}=\{0\}$. For example, $\rG$ could be a product of a compact subgroup (for $\faperp$)  and a diagonal group (for $\fa$). We need $\beta_0 <0, \beta_1 < 0$.
\item  $(\faperp)_{\asym}=\{0\}, (\fa)_{\sym}=\{0\}$. This shows $\fa=\fg_{\asym}$.
  An example is $\GL(n)$ in \cite{VAV,MartinNeff}, reviewed under our point of view with parallel transport in \cref{sec:GL}. The case of $\Spl(n)$ in \cite{Fiori} could be treated similarly. In this case $\beta_0 > 0, \beta_1>0$.\label{item:GL} Except for the case $\fg_{\sym}=\{0\}$ already considered, the trace form is indefinite, it is positive on $\fg_{\sym}$ and negative on $\fg_{\asym}$. The deformation metric \cref{eq:metCheeger} modifies the trace form on the antisymmetric part by a negative multiplier $-\beta_1$. This is also known in \cite{Barp2019}.
\item  $(\faperp)_{\asym}=\{0\}, (\fa)_{\asym}=\{0\}$. In this case, $\fg\subset\cE_{\sym}$, an example is a group of commutative symmetric matrices, or equivalently, a group of the form $U\cD U^{\sfT}$ where $U$ is a fixed orthogonal matrix and $\cD$ is a diagonal group.
\end{enumerate}  
If $\rB$ is not restricted to the form $\beta\beta_0\dI_{\fa}$, there are more possibilities, we discuss an example in \cref{sec:GL}. We will not be able to discuss an example where the bilinear form is not a trace form.
\begin{table}
\begin{tabular}{c|l|l|c}
\toprule
$\fa$ and $\faperp$  &  $\beta_0$ & $\beta_1$ & Examples \\
\midrule
 $(\faperp)_{\sym} = \{0\}$, $\fa_{\sym} = \{0\}, \fg_{\sym} =\{0\}$ & $-$ & $+$ & $\rA\subset \rG\subset \SO(n)$.\\
 $(\faperp)_{\sym} = \{0\}, \fa_{\asym} = \{0\} $ &$-$ &$-$ & $\cD =\rA \subset \rG = \SO(n)\times \cD$. \\
 $(\faperp)_{\asym} = \{0\}, \fa_{\sym} = \{0\} $ & $+$ & $+$ & $\SO(n)=\rA \subset \rG =  \GL(n)$.\\
 $(\faperp)_{\asym} = \{0\}, \fa_{\asym} = \{0\}, \fg_{\asym} =\{0\} $ & $+$&$-$ & $\rA\subset \rG = U\cD U^{\sfT}$, $UU^{\sfT}= \dI_n$.\\
\bottomrule
\end{tabular}  
\caption{Summary of Riemannian left-invariant metrics obtained by deforming a transposable algebra $\fg$ using a transposable $\fa\subset \fg$. We assume $\cD$ is a diagonal group.
}
\label{tab:beta}
\end{table}

\section{The General Linear group with positive determinant}\label{sec:GL}
In this case, the group is $\rG=\GL^+(n)$, the group of invertible matrices with positive determinant in $\R^{n\times n}$, with Lie algebra $\fg=\gl(n)=\R^{n\times n}$. We take $\fa = \oo(n)$, the Lie algebra of antisymmetric matrices. For $\ttX\in\fg$, we have $\ttpa \ttX = \ttX_{\asym}$. We typically take $\beta_0=1$, and consider $\rB = \beta\dI_{\fa}$. For $\ttX\in\fg$, the metric becomes
\begin{equation}\Tr\ttX^2 - (\beta+1)\Tr\ttX_{\asym}\ttX_{\asym} = \Tr\ttX^2_{\sym} + \beta\Tr\ttX_{\asym}^{\sfT}\ttX_{\asym} = \lvert \ttX\rvert_{\cF}^2 + (\beta-1)\lvert \ttX_{\asym}\rvert_{\cF}^2
  \label{eq:glnmetric}
\end{equation}
using the decomposition $\ttX=\ttX_{\sym} + \ttX_{\asym}$. We need $\beta > 0$ for the metric to be Riemannian. When $\beta = 1$, this is the Frobenius form, covered in \cite{VAV}.
Expanding the $\mathrm{skew}$ operator, \cref{theo:mainChris} implies
\begin{proposition}With the metric \cref{eq:glnmetric}, the Levi-Civita connection is given by
  \begin{equation}
    \Gamma(\xi, \eta) = -\frac{1}{2}(\xi X^{-1}\eta + \eta X^{-1}\xi)
     + \frac{1+\beta}{4}X([X^{-1}\xi, (X^{-1}\eta)^{\sfT}] + [X^{-1}\eta, (X^{-1}\xi)^{\sfT}]).
  \end{equation}
  The geodesic starting at $X$ with initial velocity $\xi=X\ttV$ is given by
  \begin{gather}\gamma(t) = X \exp(\frac{t}{2}\left((1-\beta)\ttV+(1+\beta)\ttV^{\sfT}\right))\exp(t(1+\beta)\ttV_{\asym}).
  \end{gather}
  Parallel transport of a tangent vector $\eta$ at $X$ along $\gamma(t)$ is given by
  \begin{equation}
    \ttT_0^t(\gamma)\eta = X\exp(\frac{t}{2}\left((1-\beta)\ttV+(1+\beta)\ttV^{\sfT}\right))\expa(t \rP,  X^{-1}\eta)\exp(t(1+\beta)(\ttV_{\asym})),
  \end{equation}
  where the operator $\rP = \rP_{\ttV}$ is given by
  \begin{equation}\rP_{\ttV}\ttB = \frac{1}{2}([\ttB,\ttV]+(1+\beta)([\ttV_{\asym},\ttB] -[\ttB_{\asym},\ttV])).
  \end{equation}
\end{proposition}
We illustrate \cref{theo:main} by discussing the naturally reductive structure of the Frobenius metric in \cite{VAV}, corresponding to $\beta=1$. The metric \cref{eq:metCheeger} for $\xi, \eta\in T_X\GL(n)$ is $\Tr(X^{-1}\xi(X^{-1}\eta)^{\sfT})$. 
The group $\trG$ is $\GL(n)\times\SO(n)$, which acts on $\rG=\GL(n)$ by $(X, A).Y = XYA^{\sfT}$, $X,Y\in\GL(n)$, $A\in \SO(n)$ (see also \cite{MartinNeff}). The lift of $\ttX\in \fg=\gl(n) = \R^{n\times n}$ in \cref{eq:horLiftIntert} is $(\ttX^{\sfT}, \ttX^{\sfT}-\ttX)$, and the horizontal space $\tilde{\fm}$ consists of vectors of this form. The metric \cref{eq:metGH} is
\begin{equation}\langle (\ttX_1, \ttA_1) , (\ttX_2, \ttA_2)\rangle_{\trG} = \Tr\ttX_1\ttX_2 - \frac{1}{2}\Tr\ttA_1\ttA_2 \text{ for }(\ttX_1, \ttA_1) , (\ttX_2, \ttA_2)\in \trG.
\end{equation}
and evaluating the inner product on the lifts shows this is a submersion to $\GL(n)$ with the Frobenius metric $(\ttX_1, \ttX_2) = \Tr\ttX_1^{\sfT}\ttX_2$. The $\ad_{\trH}$-invariance condition on $\tilde{\fm}$ implies $\ad_{\oo(n)}$-invariance on $\gl(n)$, that means for $\ttA \in \oo(n), \ttX_1, \ttX_2 \in \gl(n)$
$$\Tr [\ttA, \ttX_1]^{\sfT}\ttX_2 +  \Tr\ttX_1^{\sfT}[\ttA, \ttX_2]=0,
$$
which also follows from $\Tr [\ttA, \ttX_1]^{\sfT}\ttX_2= \Tr [\ttX_1^{\sfT}, -\ttA]\ttX_2 = - \Tr[\ttA, \ttX_2]\ttX_1^{\sfT}$. For the naturally reductive condition \cref{eq:NatReductiveInv}, the projection $\cP_{\tilde{\fm}}$ is given by $\cP_{\tilde{\fm}}(\ttX, \ttA) =  (\ttX^{\sfT}+\ttA, 2\ttA+ \ttX^{\sfT} - \ttX)$ from \cref{eq:decompose}.
Denote $\tilde{\ttX}\in\tilde{\fm}$ for the lift of $\ttX\in\gl(n)$ then for $\ttX_1, \ttX_2\in \gl(n)$
$$\cP_{\tilde{\fm}}[\tilde{\ttX}_1, \tilde{\ttX}_2] = \cP_{\tilde{\fm}}([\ttX_1^{\sfT}, \ttX_2^{\sfT}], [\ttX_1^{\sfT}-\ttX_1, \ttX_2^{\sfT}-\ttX_2]),
$$
which is the lift of $[\ttX_1^{\sfT}, \ttX_2^{\sfT}] - [\ttX_1^{\sfT}-\ttX_1, \ttX_2^{\sfT}-\ttX_2] =[\ttX_1^{\sfT}, \ttX_2] + [\ttX_1, \ttX_2^{\sfT}] - [\ttX_1, \ttX_2]$. Thus, condition \cref{eq:NatReductiveInv}, which follows from the quotient structure, would also follow from the following identity, which can be verified directly by manipulating trace identities
\begin{equation}\Tr([\ttX_1^{\sfT}, \ttX_2] + [\ttX_1, \ttX_2^{\sfT}] - [\ttX_1, \ttX_2])\ttX_3^{\sfT} = \Tr([\ttX_3^{\sfT}, \ttX_1] + [\ttX_3, \ttX_1^{\sfT}] - [\ttX_3, \ttX_1])\ttX_2^{\sfT}. 
\end{equation}

Instead of taking $\fa=\oo(n)$, if we take $\fa = \oo(n)\oplus \R\dI_n$, then the scalar operator $\beta\dI_{\fa}$ violates the condition in \cref{item:GL} of the list of possible choices of $\beta$ to produce a Riemannian metric. However, if we take instead the operator $\rB$ which acts by multiplying $\oo(n)$ by $\beta$ and $\R\dI_n$ by a negative constant $-\mu$, that is for $\ttA \in \fa$
$$\rB \ttA = \beta(\ttA - \frac{1}{n}\Tr(\ttA)\dI_n) - \frac{\mu}{n} \Tr(\ttA)\dI_n .$$
Then the metric in \cref{sec:riemannFam} is Riemannian, the formulas still hold, as $\R\dI_n$ does not contribute to the Lie brackets. This is the metric in \cite{MartinNeff}.
\section{The Special Orthogonal group with a deformation metric}\label{sec:SOSE}
In this case, $\rG=\SO(n)$, the group of orthogonal matrices with determinant $1$. Its Lie algebra $\fg = \oo(n)$ is identified with the set of antisymmetric matrices. For $\fa$, we take the algebra of matrices where the top $d\times d$ block is antisymmetric and the remaining entries vanish, ($\fa =\begin{bmatrix}\oo(d) & 0_{d \times(n-d)}\\ 0_{(n-d) \times d} & 0_{(n-d) \times (n-d)}\end{bmatrix}$ graphically). Here, the projection $\ttpa$ from $\R^{n\times n}$ maps a matrix $\ttW$ to the matrix $\ttpa \ttW$ with the top $d\times d$ block is antisymmetrized from the corresponding block of $\ttW$, while the remaining entries vanish. The standard normalization takes $\beta_0 = -\frac{1}{2}$. For $\rB = \alpha\dI_{\fa}$, we have
\begin{proposition}At $X\in \SO(n)$, $\xi, \eta\in T_X\SO(n)$ where $\xi = X\begin{bmatrix}A_{\xi} & -B^{\sfT}_{\xi}\\ B_{\xi} & C_{\xi} \end{bmatrix}$, $\eta = X\begin{bmatrix}A_{\eta} & -B^{\sfT}_{\eta}\\ B_{\eta} & C_{\xi} \end{bmatrix}$, with $A_{\xi}, A_{\eta}\in\oo(d), C_{\xi}, C_{\eta}\in \oo(n-d)$, the metric \cref{eq:metCheeger} with $\beta_0 = -\frac{1}{2}, \rB = \alpha\dI_{\fa}$ is given by
  \begin{equation}\langle\xi, \eta \rangle =\frac{1}{2}\Tr \xi^{\sfT}\eta + (\alpha-\frac{1}{2})
    \Tr(\ttpa(\xi^{\sfT}X)\ttpa(X^{\sfT}\eta))=\frac{1}{2}\Tr \xi^{\sfT}\eta + (\alpha-\frac{1}{2})
    \Tr A_{\xi}^{\sfT}A_{\eta} .\label{eq:SOmetric}
  \end{equation}
The Christoffel function of the Levi-Civita connection is  
  \begin{equation}\begin{gathered}
    \Gamma(\xi, \eta) = \frac{1}{2}
    X(\xi^{\sfT} \eta + \eta^{\sfT}\xi)+ \frac{1-2\alpha}{2}X([\ttpa(X^{\sfT}\xi), X^{\sfT}\eta] +
    [\ttpa(X^{\sfT}\eta), X^{\sfT}\xi])\\
    = \frac{1}{2}
    X(\xi^{\sfT} \eta + \eta^{\sfT}\xi)+ \frac{1-2\alpha}{2}X\begin{bmatrix}0_{d\times d} & -A_{\xi}B_{\eta}^{\sfT} - A_{\eta}B_{\xi}^{\sfT}\\
    -B_{\eta}A_{\xi} - B_{\xi}A_{\eta} & 0_{(n-d)\times (n-d)}.
    \end{bmatrix}.\label{eq:christSO}
\end{gathered}    
  \end{equation}  
The geodesic $\gamma(t)$ with $\gamma(0) = X, \dot{\gamma}(0) = \xi$ is
  \begin{gather}\label{eq:SOgamma}
    \gamma(t)= X\exp\left(t\begin{bmatrix} 2\alpha A_{\xi} & -B_{\xi}^{\sfT}\\ B_{\xi} & C_{\xi}\end{bmatrix}\right)\exp\left(t\begin{bmatrix} (1-2\alpha)A_{\xi} & 0\\ 0 & 0\end{bmatrix}\right).
  \end{gather}
  Parallel transport of $\eta$ along $\gamma(t)$ is given by
  \begin{equation}
    \ttT_0^t(\gamma)\eta = X\exp\left(t\begin{bmatrix} 2\alpha A_{\xi} & -B^{\sfT}_{\xi}\\ B_{\xi} & C_{\xi}\end{bmatrix}\right)\expa(t\rP, X^{\sfT}\eta)\exp\left(t\begin{bmatrix} (1-2\alpha)A_{\xi} & 0\\ 0 & 0\end{bmatrix}\right)\label{eq:parSON}
  \end{equation}
  where the operator $\rP=\rP_{\ttV}$ for $\ttV = X^{\sfT}\xi$ is given in \cref{eq:transportCheeger}.
\end{proposition}
This is the metric in \cite{ExtCurveStiefel} for the $\SO(n)$ identified with the Stiefel manifold $\St{n}{n}$.
\section{A naturally reductive family of quotient metrics}\label{sec:submerse}
We now introduce a deformation metric on a quotient manifold $\rG/\rK$ where $\rG$ has a deformation metric by a subgroup $\rA$. The resulting metric is also naturally reductive. The two standard metrics on the Stiefel manifold are instances of these metrics. Roughly, the requirement is the Lie algebra of $\rK$ is both orthogonal and commuting with the Lie algebra $\fa$ of $\rA$. The precise statement is in \cref{theo:hornatural} below.

Let us explain through an example of a space of block matrices. For $k < n$, let $\fa$ be a copy of $\R^{k\times k}$, considered as a subspace of $\cE=\R^{n\times n}$ by identifying $\fa$ with the space of upper-left $k\times k$ block diagonal matrices, and $\faperp$ its Frobenius orthogonal complements. Hence, $\fa$ consists of matrices of the form $\begin{bmatrix} A &0\\0& 0\end{bmatrix}$ with $A\in\R^{k\times k}$, while $\faperp$ consists of those of the form $\begin{bmatrix} 0 & B\\C& D\end{bmatrix}$, with $D\in \R^{(n-k)\times (n-k)}, B\in \R^{k\times (n-k)}, C\in \R^{(n-k)\times k} $. Notice that $\fa$ and $\fk = \{\begin{bmatrix} 0 & 0\\0& D\end{bmatrix}|D\in \R^{(n-k)\times (n-k)}\}$ are both commuting and orthogonal. This is the situation we are interested in.
\begin{theorem}\label{theo:hornatural}Assume $\rA\subset \rG$ are connected matrix Lie groups with Lie algebras $\fa\subset\fg$, with a bi-invariant symmetric bilinear form $\langle,\rangle$ on $\fg$ and $\rB$ is an intertwining operator on $\fa$ as in \cref{theo:main}. Let $\rK$ be a closed Lie group commuting with $\rA$, with a nondegenerate Lie algebra $\fk$ of $\rK$ satisfying 
\begin{equation}
[\fk, \fa] = \{0\}\text{ and } \langle \fk, \fa \rangle =\{0\}.\label{eq:top}
\end{equation}
Then $\rK$ acts by right multiplication as a group of isometries on $\rG$ with the metric \cref{eq:metCheeger}. The resulting quotient space $\cM=\rG/\rK$ with the quotient metric is naturally reductive when identified with $\trG/\trK$, where $\trG=\rG\times \rA$ and 
\begin{equation}\trK = \{ (KH, H)\in \trG | K\in \rK, H\in \rA\}.
\end{equation}
Specifically, let $\fm$ be the orthogonal complement of $\fk$ in $\langle, \rangle$. Then $\fa\subset \fm$. Set
\begin{equation}\begin{gathered}\tfk = \{(\ttK + \ttA, \ttA)| \ttK\in \fk, \ttA\in\fa\},\\
\tfmfq = \{(\ttM, \hrB\ttpa\ttM)| \ttM\in \fm\} \text{ with }\hrB =  (\rB+\beta_0\dI_{\fa})\rB^{-1}. \label{eq:reductiveComp}
\end{gathered}
\end{equation}
Then $\tfg =\tfmfq \oplus \tfk$, with the reductive decomposition for $(\ttX, \ttA)\in \tfg$ is given by
\begin{equation}\begin{gathered}
(\ttX, \ttA) =  \left(\cP_{\fm}\ttX
- (\dI_{\fa}+\beta_0^{-1}\rB)\ttpa\ttX + \beta_0^{-1}\rB\ttA
, -(\dI_{\fa}+\beta_0^{-1}\rB)\ttpa(\ttX-\ttA)\right)\\
+ \left(\cP_{\fk}\ttX + \ttA + (\dI_{\fa}+\beta_0^{-1}\rB)\ttpa(\ttX-\ttA) , \ttA + (\dI_{\fa}+\beta_0^{-1}\rB)\ttpa(\ttX-\ttA)\right).\label{eq:decomposeK}
\end{gathered}
\end{equation}
The metric on $\trG/\trK$ given by \cref{eq:metricM} restricted to $\tfmfq$ corresponds to \cref{eq:metCheeger}.

When $\dI_{\fa} + \beta^{-1}\rB$ is invertible, if $\trG$ is equipped with the bi-invariant metric \cref{eq:metGH} then $\cM$ is isometric to $\trG/\trK$. 
Let $\pi:\rG\mapsto \rG/\rK$ be the quotient submersion and $\fq:\trG\mapsto \rG$ be the submersion in \cref{theo:main}. Then the map 
\begin{equation}\fqh = \pi\circ \fq: \trG\mapsto \rG/\rK=\cM,\quad\quad \fqh(X, A) = \lb XA^{-1} \rb \quad\text{ for }(X, A) \in \trG \end{equation}
is a pseudo-Riemannian submersion. Here, $\lb XA^{-1} \rb\in \cM = \rG/\rK$ is the equivalence class of $XA^{-1}$ defined by the right multiplication by $\rK$.
\end{theorem}
\begin{proof}
For $K\in \rK$, $\nu\in \fg$ and $\ttA\in \fa$, differentiating $K^{-1}\expm(t\ttA) K =  \expm(t\ttA)$, we get $K^{-1}\ttA K =  \ttA$. Using $\Ad_K$ invariance of $\langle,\rangle$
$$\begin{gathered}\langle \ttA, K^{-1}\nu K\rangle = \langle K^{-1}\ttA K, K^{-1}\nu K\rangle = \langle \ttA, \nu\rangle = \langle \ttA , \ttpa \nu\rangle.
\end{gathered}$$
Thus, $\ttpa K^{-1}\nu K = \ttpa \nu$.
For $\eta = X\nu$ in $T_X\rG$ with $\nu\in\fg$, right multiplication by $K$ sends $X$ to $XK$ and $\eta$ to $\eta K$. By left invariance of $\langle,\rangle_{\vbeta}$ and $\Ad$-invariance of $\langle,\rangle$,
if $\vbeta_{XK}\langle , \rangle$ denotes the pseudo-Riemannian product $\langle, \rangle_{\vbeta}$ at $XK$
  $$\begin{gathered}\vbeta_{XK}\langle \eta K, \eta K\rangle = \vbeta_{XK}\langle X\nu K, X\nu K\rangle=\langle (XK)^{-1}X\nu K, (XK)^{-1} X\nu K\rangle_{\vbeta} \\
=   \beta_0 \langle K^{-1}\nu K, K^{-1}\nu K\rangle - \beta_0\langle (\dI_{\fa}+\beta_0^{-1}\rB) \ttpa (K^{-1}\nu K), \ttpa(K^{-1}\nu K)\rangle\\
 = \beta_0\langle  \nu ,\nu \rangle -\beta_0\langle (\dI_{\fa}+\beta_0^{-1}\rB) \ttpa\nu, \ttpa\nu\rangle = \langle \nu, \nu\rangle_{\vbeta}.
\end{gathered}$$
Thus, $\rK$ acts as a group of isometries on $\rG$ by right multiplication. Hence, $\pi$ is a submersion from $\rG$ if $\cM$ is equipped with the quotient metric. Since $YKA^{-1} = YA^{-1}K$ for $Y\in \rG, K\in \rK$ and $A\in \rA$, there is a well-defined action by $\trG$ on $\rG/\rK$ by $(X, A): \lb Y\rb\mapsto \lb XYA^{-1}\rb$, and $\trG$ acts as a group of isometries. The inverse image of $\lb XA^{-1} \rb$ consists of $(X_1, A_1)$ such that $X_1A_1^{-1} = XA^{-1}K$ for $K\in \rK$. Set $H = A^{-1}A_1$ then $A_1=AH$,
$$X_1 = XA^{-1}KA_1=XA^{-1}A_1K = XHK= XKH$$
since $K$ and $H$ commute. Thus, $(X_1, A_1) = (X, A) (KH, H)\in (X, A)\trK$. Since $\rK$ and $\rH$ are closed subgroups, $\trK$ is closed, identified with the isotropy group at $\lb \dI_{\rG}\rb$. From \cite[Proposition 11.12]{ONeil1983}, we can identify $\cM=\rG/\rK$ with $\trG/\trK$.

We now show $\tfk\cap \tfmfq = \{0\}$. In \cref{eq:reductiveComp}, for a point in the intersection, $\ttK = 0, \ttA=\ttM$ and $\ttM =\hrB\ttpa\ttM$, thus $\ttA=\ttM=0$. We confirm that we have the equality in \cref{eq:decomposeK}, and the two terms in the decomposition are in $\tfm$ and $\tfk$, respectively, giving us the reductive decomposition, where $\Ad_{\tfk}$ invariance will follow from the next argument.

We argue as before that we only need to prove the remaining naturally reductive conditions for nonsingular $\dI_{\fa}+\beta_0^{-1}\rB$, which follows from bi-invariance of the metric \cref{eq:metGH} on $\trG$. The singular case follows by continuity as the pairing on $\tfmfq$ is well-defined and continuous in the parameters. The horizontal lift is still given by \cref{eq:horLiftIntert}, restricted to $\ttZ\in \fm$. Thus, most requirements have been verified in the proof of \cref{prop:main}. It remains to show \cref{eq:decomposeK} is an orthogonal decomposition with the metric \cref{eq:metGH} when $\dI_{\fa}+\beta_0^{-1}\rB$ is nonsingular, which is clear. Therefore, the induced metric on $\cM$ is naturally reductive.
\end{proof} 

\begin{lemma} \label{lem:faaftop}For an intertwining operator $\rB$ on $\fa\subset \fg$ with a bi-invariant form $\langle,\rangle$, for  $\ttX_1, \ttX_2 \in \fg$, assume the Lie algebra $\fk$ satisfies \cref{eq:top}. Then we have
\begin{equation}\begin{gathered} \langle \fa,  [\rB\ttpa \ttX_1, \ttX_2] + [\rB\ttpa \ttX_2, \ttX_1]\rangle = \{0\},\\
\langle \fk,  [\rB\ttpa \ttX_1, \ttX_2] + [\rB\ttpa \ttX_2, \ttX_1]\rangle = \{0\}.
\end{gathered}
\end{equation}
\end{lemma}
\begin{proof}
This follows since for $\ttA\in \fa$,
$$\begin{gathered}
\langle \ttA,  [\rB\ttpa \ttX_1, \ttX_2] + [\rB\ttpa \ttX_2, \ttX_1]\rangle = 
\langle    [\ttA, \rB\ttpa \ttX_1], \ttX_2\rangle  + \langle [\ttA, \rB\ttpa \ttX_2], \ttX_1\rangle \\
= \langle    [\ttA, \rB\ttpa \ttX_1], \ttpa\ttX_2\rangle  + \langle [\ttA, \rB\ttpa \ttX_2], \ttpa\ttX_1\rangle 
= \langle    \ttA, [\rB\ttpa \ttX_1, \ttpa\ttX_2]\rangle  + \langle \ttA, [\rB\ttpa \ttX_2, \ttpa\ttX_1]\rangle \\
= \langle    \ttA, \rB[\ttpa \ttX_1, \ttpa\ttX_2]\rangle  + \langle \ttA, \rB[\ttpa \ttX_2, \ttpa\ttX_1]\rangle =0.\\
\end{gathered}$$
We have $\langle \fk,  [\rB\ttpa \ttX_1, \ttX_2] + [\rB\ttpa \ttX_2, \ttX_1]\rangle =\{0\}$  by bi-invariance of $\langle, \rangle$ and $[\fa, \fk] =0$.
\end{proof}

From the decomposition $\fg = \fm\oplus\fk$, the horizontal space at $X\in\rG$ is $\cH_X=X\fm$, while the vertical space is $\cV_X = X\fk$. We now have a quotient version of \cref{theo:mainChris}.
\begin{theorem}\label{theo:mainChrisHor}With the assumption and the metric of \cref{theo:hornatural}, the horizontal Christoffel function $\Gamma^{\cH}$ for two horizontal vectors $\xi, \eta$ at $X\in\rG$ is given by
  \begin{equation}\begin{gathered}\Gamma^{\cH}(\xi, \eta) =
-\frac{1}{2}\left(\xi X^{-1}\eta +\eta X^{-1}\xi + X(\cP_{\fk}[X^{-1}\xi,X^{-1}\eta])\right)\\
    + \frac{1}{2}X\left([(\dI_{\fa} + \beta_0^{-1}\rB)\ttpa(X^{-1}\xi), X^{-1}\eta] + [(\dI_{\fa} + \beta_0^{-1}\rB)\ttpa(X^{-1}\eta), X^{-1}\xi]\right)\end{gathered}\label{eq:maingammahor}
  \end{equation}
if we express the horizontal connection for two horizontal vector fields $\ttY, \ttZ$ as $\nabla^{\cH}_{\ttY}\ttZ = \rD_{\ttX}\ttY + \Gamma^{\cH}(\ttX, \ttY)$. The horizontal geodesic $\gamma$ with $\gamma(0) = X, \dot{\gamma}(0) = \xi$, where $\xi$ is a horizontal vector is given by \cref{eq:geodesicCheeger}.

Let $\ttV = X^{-1}\xi\in \fm$. The parallel transport of $\eta\in \cH_X$ is given by
  \begin{gather}
      \ttT(\gamma)_{0}^t\eta = X\expm(t\left(\ttV-(\dI_{\fa}+\beta_0^{-1}\rB)\ttpa\ttV \right))\expa(t\rP_{\ttV}, X^{-1}\eta) \expm(t(\dI_{\fa}+\beta_0^{-1}\rB)\ttpa\ttV) \label{eq:transporthor}\\
\text{ where }      \rP_{\ttV}: \ttB \mapsto \frac{1}{2}\left( \cP_{\fm}[\ttB, \ttV] + ([(\dI_{\fa}+\beta_0^{-1}\rB) \ttpa\ttV, \ttB]-[(\dI_{\fa}+\beta_0^{-1}\rB) \ttpa\ttB, \ttV])\right).\label{eq:transporthor2}
\end{gather}  
$\rP_{\ttV}$ maps $\fm$ to itself and the operator $\rP_{\ttV}$ is antisymmetric in the pairing $\langle\rangle_{\vbeta}$ on $\fm$.
\end{theorem}
\begin{proof}The proof is similar to the proof of \cref{theo:mainChris}. We will use the notations and results from that proof.

We will assume $X=\dI_{\rG}$ then use left invariance. For two horizontal vectors $\xi_0, \eta_0\in \fm$ at $X=\dI_{\rG}$, consider the horizontal vector fields $\bxi: X\mapsto X\xi_0, \bareta: X\mapsto X\eta_0$, then by \cite[Lemma 7.45]{ONeil1983}, if $\rH$ is the horizontal projection and $\Gamma$ is the Christoffel function of Levi-Civita connection in \cref{eq:ChristoffelCheeger} then
$$\begin{gathered} \Gamma^{\rH}_{X=I_{\rG}}(\bxi, \bareta) = \rH(\rD_{\bxi}\bareta + \Gamma(\bxi, \bareta) ) - \rD_{\bxi}\bareta \\
= \cP_{\fm}(\xi_0\eta_0 - \frac{1}{2}(\xi_0\eta_0 + \eta_0\xi_0)) + \frac{1}{2}\left([(\dI_{\fa} + \beta_0^{-1}\rB)\ttpa\xi_0, \eta_0] + [(\dI_{\fa} + \beta_0^{-1}\rB)\ttpa \eta_0, \xi_0]\right)-\xi_0\eta_0
\end{gathered}$$
because by \cref{lem:faaftop}, the sum of the two Lie brackets in the third and fourth terms above are in $\fm$. We collect the remaining terms to
$$\begin{gathered}
    \frac{1}{2}\cP_{\fm}[\xi_0, \eta_0] - \xi_0\eta_0 = 
    \frac{1}{2}[\xi_0, \eta_0] - \frac{1}{2}\cP_{\fk}[\xi_0, \eta_0] - \xi_0\eta_0 =
 - \frac{1}{2}(\xi_0\eta_0 + \eta_0\xi_0)  - \frac{1}{2}\cP_{\fk}[\xi_0,\eta_0].
\end{gathered}$$
Together, we get $\Gamma^{\rH}_{X=I_{\rG}}(\bxi, \bareta)$ evaluated at $X = \dI_{\rG}$ is given by \cref{eq:maingammahor}. We get the general case by using left invariance.

The result on the horizontal lift of geodesics is standard.  For parallel transport, we will assume again $X=\dI_{\rG}$, the general result will follow by left invariance. 
As before, by \cref{theo:hornatural} and \cite[Proposition 2.12]{SmithThesis}, the lifted parallel transport from $(\dI_{\rG}, \dI_{\rA})$ is $(\exp(t(\xi+\xi^{\rA}))\ttC^{\rG}, \exp(t\xi^{\rA})\ttC^{\rA})$ where $(\ttC^{\rG}(t), \ttC^{\rA}(t))\in\tilde{\fm}_{\fqh}$ solves
$$\begin{gathered}(\dot{\ttC}^{\rG}, \dot{\ttC}^{\rA} ) = \cP_{\tilde{\fm}_{\fqh}} (-\frac{1}{2}[\xi + \xi^{\rA},  \ttC^{\rG}], -\frac{1}{2}[\xi^{\rA}, \ttC^{\rA}]),\\
(\ttC^{\rG}(0), \ttC^{\rA}(0)) = (\eta + \eta^{\rA}, \eta^{\rA}).
\end{gathered}$$
The lifted transport maps to $\exp(t(\xi + \xi^{\rA}))(\ttC^{\rG} - \ttC^{\rA} ) \exp(- t\xi^{\rA})$, the horizontal transport on $\rG$, and $F(t) = \ttC^{\rG}(t) - \ttC^{\rA}(t)$ solves
$$\begin{gathered}\dot{F} = -\frac{1}{2}\cP_{\fm}\left([\xi, F] - [\xi, (\dI_{\fa} + \beta_0^{-1}\rB)\ttpa F] - [(\dI_{\fa} + \beta_0^{-1}\rB)\ttpa\xi, F]\right).
\end{gathered}$$
The additional projection $\cP_{\fm}$ is from the map $d\pi$ in $d\fqh =  d\pi\circ d\fq$. Thus, $F(t)=\expa(t\rP_{\xi}, \eta)$, giving us the parallel transport formula \cref{eq:transporthor} for $X=\dI_{\rG}$, where we note that the sum of the last two brackets are in $\fm$ by \cref{lem:faaftop}. The remaining statements are immediate or follow from similar arguments in the proof of \cref{theo:mainChris}.
\end{proof}

\section{Stiefel manifolds}\label{sec:stiefel}
The Stiefel manifold $\St{n}{d}$ consists of orthogonal matrices in $\R^{n\times d}$, where we assume $d < n$. Thus, $Y\in \St{n}{d}$ if $Y^{\sfT}Y = \dI_d$. The simplest Stiefel manifolds are unit spheres, corresponding to the case $d=1$. The tangent space to $\St{n}{d}$ at $Y$ is defined by the equation $(Y^{\sfT}\xi)_{\sym}=0$, $\xi\in \R^{n\times d}$.

\subsection{An isometry and a totally geodesic submanifold}\label{sec:isometry}
In a sphere, great circles are geodesics. Given a point $Y$ on the unit sphere, and a tangent vector $\xi$, the plane spanned by $Y$ and $\xi$ cuts a great circle. Changing to the coordinate defined by $Y$ and the unit vector in the direction $\xi$, this great circle maps isometrically to the unit circle in $\R^2$. There is a similar picture for Stiefel manifolds. We start the following lemma,
with ideas in \cite{MataigneZimmermannMiolane} and discussed partially in \cite[section 5.1]{NguyenGeodesic}.
\begin{lemma}Let $Y$ be a point on the Stiefel manifold $\cM=\St{n}{d}\subset \R^{n\times d}$, and $\xi$ be a tangent vector to $\cM$ at $Y$. If $\xi\neq YY^{\sfT}\xi$ and $Q_{\xi}$ is an orthonormal basis of the column span of $\xi-YY^{\sfT}\xi$, then $\xi = YA + Q_{\xi}R_{\xi}$ for some matrix $R_{\xi}$. More generally, assume $Q$ is an orthogonal matrix in $\R^{n\times k}$ for an integer $k\geq 0$, such that $Y^{\sfT}Q = 0$ and $\xi = YA+QR$ for a matrix $R\in\R^{k\times d}$. Then the vector subspace $\cS=\cS_Q\subset \R^{n\times d}$ of matrices of the form
  \begin{equation}\cS := \{ V\in\R^{n\times d}|\; V = YV_Y + Q V_{Q}\text{ for } V_Y\in \R^{d\times d}, V_{Q}\in \R^{k\times d}\}\label{eq:subspaceS}
  \end{equation}
  intersects $\St{n}{d}$ along a submanifold $\cM_{\cS}\subset\cM$. Let $\ttF$ be the map from $\R^{n\times d}$ to $\R^{(d+k)\times d}$ defined by
  \begin{equation}  \ttF: V\mapsto FV, \quad V\in \R^{n\times d}\quad \text{ for } F = F_{Y,Q} = \begin{bmatrix}Y^{\sfT}\\Q^{\sfT}\end{bmatrix}.\label{eq:stfisometry}
\end{equation}    
Then $\ttF$ restricted to $\cM_{\cS}$ is a bijection onto $\St{d+k}{d}\subset \R^{(d+k) \times d}$.  We have $\ttF Y =\dI_{d+k,d}= \begin{bmatrix}\dI_d\\ O_{k\times d}\end{bmatrix}\in \St{d+k}{d}\subset\R^{n\times(d+k)}$.\hfill\break
If $\xi = YY^{\sfT}\xi$, we can consider $F_{Y,Q}= Y^{\sfT}$, or $Q=0$ and the intersection $\cM_{\cS}$ is isometric to $\OO(d)$, with two connected components.\label{lem:isometry}
\end{lemma}
\begin{proof}
  First, if $\xi\neq YY^{\sfT}\xi$, since columns of $Q_{\xi}$ are in the column span of $\xi - YY^{\sfT}\xi$, we have $Y^{\sfT}Q_{\xi}=0$, and $Q_{\xi}^{\sfT}Q_{\xi}=\dI_k$ by the orthonormal assumption. Express the columns of $\xi - YY^{\sfT}\xi$ in the basis $Q_\xi$ as $\xi - YY^{\sfT}\xi=Q_{\xi}R_{\xi}$ for some matrix $R_{\xi}$, then $\xi = Y(Y^{\sfT}\xi)+Q_{\xi}R_{\xi}$. This shows $Q_{\xi}$ has the desired property. 
    
  Assume $Y^{\sfT}Q=0, Q^{\sfT}Q=\dI_k$ and $\xi=YA+QR$ for some matrix $A, R$. With  $F=\begin{bmatrix}Y|Q\end{bmatrix}^{\sfT}$, we have $FF^{\sfT} = \dI_{d+k,d+k}$. If $V = YV_Y + QV_Q\in \cM_{\cS}$, then $\ttF V = FV=\begin{bmatrix}V_Y\\V_Q\end{bmatrix}$. The condition $V^{\sfT}V= \dI_d$ implies $V_Y^{\sfT}V_Y +V_Q^{\sfT}V_Q = \dI_d$ by substitution, but this means $(\ttF V)^{\sfT}(\ttF V)=\dI_d$. Thus, $\ttF$ maps $\cM_{\cS}$ to $\St{d+k}{d}$. Conversely, $V = F^{\sfT}(\ttF V)$ and the map $W\mapsto F^{\sfT}W$ maps $\St{d+k}{d}$ to $\cM_{\cS}$, hence, $\ttF$ is a bijection from $\cM_{\cS}$ to $\St{d+k}{d}$.

By direct calculation, $F Y = \dI_{d+k,d}$ and $F\xi = \begin{bmatrix}A\\R\end{bmatrix}\in T_{\dI_{d+k,d}}\St{d+k}{d}$.
  
The case $k=0$ follows similarly.    
\end{proof}
We will see the metric induced from the family in \cref{sec:SOSE} could be expressed as   
\begin{equation}   \langle \xi, \xi \rangle_{\vbeta} = \Tr\xi^{\sfT}\xi + (\alpha-1)\Tr\xi^{\sfT}YY^{\sfT}\xi,\label{eq:StiefelMetric}
\end{equation}
which is the metric \cite{ExtCurveStiefel}, using the parametrization in \cite{NguyenGeodesic}. For a tangent vector $\nu$ at $V\in\cM_{\cS}$, write $\nu= F^{\sfT}\ttF(\nu), V = F^{\sfT}\ttF(V)$ and substitute, we get
$$   \langle \nu, \nu \rangle_{\vbeta} = \Tr\ttF(\nu)^{\sfT}\ttF(\nu) + (\alpha-1)\Tr\ttF(\nu)^{\sfT}\ttF(V)\ttF(V)^{\sfT}\ttF(\nu).
$$
Thus, if we equip $\St{d+k}{d}$ with a metric with the same parameter $\alpha$, $\ttF$ is an isometry.

Computing on this smaller Stiefel manifold is more effective, as observed in the case of geodesics \cite{Edelman_1999,NguyenGeodesic,ZimmermanHuper,Rentmee,Zimmer}. The case of parallel transport will be done similarly. Besides the original paper \cite{Edelman_1999}, the first author benefited from the prior works \cite[Chapter 5]{Rentmee} and \cite{Zimmer} in understanding the advantage of working on the \emph{baby problem}, as referred to in the first citation.

By \cite[page 104]{ONeil1983}, $\cM_{\cS}$ is a totally geodesic submanifold, as the geodesic with initial point $V\in \cM_{\cS}$ and initial vector tangent to $\cM_{\cS}$ will stay on $\cM_{\cS}$.
\subsection{Submersion metric, connection and parallel transport}\label{sec:submerseStiefel}The material on the submersion metric and the connection of this section follows from the treatment of Stiefel manifolds in the original paper \cite{Edelman_1999}, the thesis \cite{Rentmee} (both with $\alpha\in \{\frac{1}{2}, 1\}$), and  \cite{VMSil,ExtCurveStiefel}, where the quotient structure $(\rG\times \rA)/\trK$ with $\rG = \SO(n), \rA=\SO(d), \rK=\SO(n-d)$ and $\trK$ as in \cref{theo:hornatural} is covered with a different metric parametrization.

The Stiefel manifold $\St{n}{d}$ could be considered as a quotient manifold of $\SO(n)$ under the map sending an $n\times n$ matrix to its first $d$ columns, or $\Theta: X\mapsto X\dI_{n, d}$, for $X\in \SO(n)$, $\dI_{n,d}=\begin{bmatrix}\dI_d\\ 0_{(n-d)\times d}\end{bmatrix}$. This map is surjective, realizing $\St{n}{d}$ as a quotient of $\SO(n)$ by the group $\rK$ of block diagonal matrices of the form $\diag(\dI_d, Q)$ where $Q\in \SO(n-d)$. This is due to the fact that we can complement an orthogonal matrix $Y$ with an orthogonal matrix $Y_{\perp}\in \St{n}{n-d}$ to an element of $\SO(n)$, and given $Y_{\perp}$, any other complement basis is of the form $Y_{\perp}Q$ for $Q\in\rK$. In this case, $\fa$ is the top $d\times d$ block-matrix copy of $\oo(d)$ and $\fk$ is the bottom $(n-d)\times (n-d)$ copy of $\oo(n-d)$. The complement $\fm$ consists of all matrices with the lower right $(n-d)\times(n-d)$ blocks  identically zero. Thus, graphically, $\fk = \begin{bmatrix}0_{d\times d} & 0 \\ 0 &\oo(n-d)\end{bmatrix}, \fm = \begin{bmatrix}\fa  & * \\ * & 0_{(n-d)\times (n-d)}\end{bmatrix}$.

The differential $d\Theta$ at $\dI_n$ also maps a Lie algebra element in $\oo(n)$ to its first $d$ columns. If $\Delta = \begin{bmatrix} A\\B\end{bmatrix}$ is a tangent vector at $\dI_{n,d}$, then its horizontal lift is the matrix $\bar{\Delta} = \begin{bmatrix} A & -B^{\sfT}\\B & 0\end{bmatrix}$, the bottom right entry corresponding to $\fk$ is set to zero. With $\beta_0=-\frac{1}{2}, \rB = \alpha\dI_{\fa}$, the metric in \cref{eq:SOmetric} at $\dI_n$ is
      $$\langle \bar{\Delta}, \bar{\Delta} \rangle_{\vbeta} = \frac{1}{2}(\Tr A^{\sfT}A + 2\Tr B^{\sfT}B) + (\alpha-\frac{1}{2})\Tr A^{\sfT}A = \alpha\Tr A^{\sfT}A + \Tr B^{\sfT}B.
    $$
More generally, for $\bX\in \SO(n)$, split $\bX=\begin{bmatrix}Y|\Yperp\end{bmatrix}$ along the $d$-th column, if $\bxi\in T_{\bX}\SO(n)$ is horizontal and maps to $\xi\in T_Y\St{n}{d}$, with $\bX^{\sfT}\bxi=\bar{\Delta}$ above then we deduce $A = Y^{\sfT}\xi, B=\Yperp^{\sfT}\xi$. Using $\Yperp\Yperp^{\sfT} + YY^{\sfT} = \dI_n$, the metric is as given by \cref{eq:StiefelMetric}
$$\alpha \Tr \xi Y Y^{\sfT}\xi + \Tr \xi \Yperp \Yperp^{\sfT}\xi =
\Tr\xi^{\sfT}\xi + (\alpha - 1)\Tr \xi^{\sfT}Y Y^{\sfT}\xi.
$$
Moreover, set $\dIperp = \begin{bmatrix} 0_{d\times(n-d)} \\ \dI_{n-d}\end{bmatrix}$ then the horizontal lift $\bxi$ is given by
\begin{equation}\bxi = \xi\dI_{n,d}^{\sfT} - Y\xi^{\sfT}\Yperp\dIperp^{\sfT}\label{eq:stiefellift}.
\end{equation}
To see this, with $\bxi$ as above, we have $\bxi\dI_{n,d}=\xi$ and using $Y^{\sfT}Y = \dI_d, \Yperp^{\sfT}Y = 0$ then
$$\bX^{\sfT}\bxi= (Y\dI_{n,d}^{\sfT} + \Yperp\dIperp^{\sfT})^{\sfT}\bxi=
\dI_{n,d}Y^{\sfT}\xi\dI_{n,d} - \dI_{n,d}\xi^{\sfT}\Yperp\dIperp^{\sfT} +\dIperp\Yperp^{\sfT}\xi\dI_{n,d}^{\sfT},$$
which is antisymmetric, and $\dIperp^{\sfT}(X^{\sfT}\bxi)\dIperp =0$. We also have $\ttpa(X^{\sfT}\bxi)= \dI_{n,d}Y^{\sfT}\xi\dI_{n,d}^{\sfT}$.

Under this metric and with the lift above, using \cref{eq:maingammahor}, we can show \cite{ExtCurveStiefel,NguyenLie} the Levi-Civita connection for Stiefel manifold is $\nabla_{\ttZ}\ttW = \rD_{\ttZ}\ttW + \Gamma(\ttZ, \ttW)$ with the Christoffel function evaluated at $Y\in \cM=\St{n}{d}, \ttZ(Y) =\xi, \ttW(Y) =\eta$ is
\begin{equation}
  \label{eq:stiefel_connect}
\Gamma(\xi, \eta) = \frac{1}{2}Y(\xi^{\sfT}\eta+\eta^{\sfT}\xi) +(1-\alpha)(\dI_n-YY^{\sfT})(\xi\eta^{\sfT}+\eta\xi^{\sfT})Y.
\end{equation}

For $Y\in \cM=\St{n}{d}$, with $\Theta(\bX)=Y$, and $\xi\in T_Y\cM= T_Y\St{n}{p}$, using \cref{eq:geodesicCheeger}, the geodesics, with initial conditions $\gamma(0) = Y, \dot{\gamma}(0) = \xi$ is
\begin{equation}\gamma(t) = \bX\expm(t\begin{bmatrix}2\alpha A & -B^{\sfT}\\ B & 0 \end{bmatrix} )\expm(t\begin{bmatrix}(1-2\alpha)A & 0\\0 & 0 \end{bmatrix})\dI_{n,d}
\end{equation}
if $\bxi = \bX\begin{bmatrix}A & -B^{\sfT}\\ B & 0 \end{bmatrix}$. Parallel transport could also be computed using \cref{eq:transporthor} as $\rK$ commutes with $\rA$. The operations are $O(n^3)$ and not efficient.

However, we can use the isometry in \cref{sec:isometry}. For geodesics, the isometry gives us the geodesic formula $\gamma(t) = F^{\sfT}\gamma_{\ttF\cM_{\cS}}(t)$, where $\gamma_{\ttF\cM_{\cS}}$ is the geodesic on $\St{d+k}{d}$ starting at $\dI_{d+k,d}$ with the initial velocity $\begin{bmatrix}A\\R\end{bmatrix}$. We lift $\gamma_{\ttF\cM_{\cS}}$ to $\gamma_{\ttF\cM_{\cS}}^{\SO(d+k)}\in \SO(d+k)$ starting at $\dI_{d+k}$, with the initial horizontal velocity $\begin{bmatrix}A & -R^{\sfT}\\ R & 0 \end{bmatrix}$. The resulting expression $F^{\sfT}(\gamma^{\SO(d+k)}_{\ttF\cM_{\cS}}(t)\dI_{d+k,d})$ is the known geodesic formula in \cite{Edelman_1999,ZimmermanHuper,NguyenGeodesic}
  \begin{equation}
    \label{eq:st_geo2}
    \gamma(t) = \begin{bmatrix}Y|Q\end{bmatrix}\expm(t\begin{bmatrix}2\alpha A & -R^{\sfT}\\ R & 0 \end{bmatrix} )\begin{bmatrix}\expm((1-2\alpha)t\bA)\\ 0_{k,d}\end{bmatrix}.
\end{equation}
  The parallel transport equation \cref{eq:transporthor} could be used for tangent vectors of $\cM_{\cS}$. Since parallel transports are linear in initial vectors, for the complete picture, we will need a transport formula for tangent vectors to $\St{n}{d}$ normal to $T_Y\cM_{\cS}$. We have
  \begin{lemma}\label{lem:normal}For a geodesic $\gamma(t)$ with $\gamma(0) =Y\in \cM=\St{n}{p}, \dot{\gamma}(0) = \xi\in T_Y\cM$, let $\cS$ be the vector subspace of $\R^{n\times d}$ of matrices with columns in the span of columns of $Y$ and $\xi$ as in \cref{eq:subspaceS}. If $\eta\in T_Y\St{n}{d}$ satisfies both $\eta^{\sfT}Y = 0$ and $\eta^{\sfT}\xi = 0$, or equivalently, $\eta$ is normal to $\cS$ in the metric \cref{eq:StiefelMetric}, then with $A = Y^{\sfT}\xi$, the parallel transport of $\eta$ along $\gamma$ in that metric is given by
    \begin{equation}\ttT(\gamma)_0^t\eta = \Delta(t) := \eta \expm(t(1-\alpha)A).\label{eq:normalTransport}
\end{equation}      
\end{lemma}    
  \begin{proof}First, we show the condition $\eta$ is normal to $\cS$ is equivalent to $Y^{\sfT}\eta = 0$ and $\xi^{\sfT}\eta = 0$. We need to show $\langle \eta, YA_1 +QR_1\rangle_{\vbeta} = 0$ for all antisymmetric $A_1$ and $R_1\in \R^{k\times d}$ means $\eta^{\sfT}Y = 0$ and $\eta^{\sfT}Q =0$. But
    $\langle \eta, YA_1\rangle_{\vbeta} = \alpha\Tr \eta^{\sfT}Y A_1$ and $\langle \eta, QR_1\rangle_{\vbeta} = \Tr \eta^{\sfT}Q R_1$
and $\eta^{\sfT}Y$ is antisymmetric. Since the Frobenius form is nondegenerate on the space of antisymmetric matrices and on $\R^{k\times d}$, we have $\eta^{\sfT}Y =0$ and $\eta^{\sfT}Q=0$. The reverse direction is clear.

The parallel transport equation is $\dot{\Delta} + \Gamma(\gamma(t); \dot{\gamma}(t), \Delta(t)) = 0$. With $\gamma$ in \cref{eq:st_geo2}
$$\dot{\gamma}(t) = \begin{bmatrix}Y|Q\end{bmatrix}\expm(t\begin{bmatrix}2\alpha A & -R^{\sfT}\\ R & 0 \end{bmatrix} )
\begin{bmatrix}A & -R^{\sfT}\\ R & 0 \end{bmatrix}
\begin{bmatrix}\expm((1-2\alpha)t\bA)\\ 0_{k,d}\end{bmatrix}.
$$
Since $\gamma(t)$ and $\dot{\gamma(t)}$ stay in $\cS$, they are are normal to $\eta$. If $\Delta$ is given by \cref{eq:normalTransport} then $\Delta^{\sfT}\gamma = \Delta^{\sfT}\dot{\gamma}=0$ since for $S\in\cS$, $\Delta^{\sfT}S = \expm(-t(1-\alpha)A)\eta^{\sfT}S=0$. Substitute in \cref{eq:stiefel_connect}, we get $\Gamma(\dot{\gamma},\Delta) = (1-\alpha)\Delta \dot{\gamma}^{\sfT}\gamma$ and
$$\begin{gathered}
\dot{\Delta}+\Gamma(\dot{\gamma},\Delta)=\dot{\Delta} + (1-\alpha)\Delta \begin{bmatrix}\expm((1-2\alpha)t\bA)\\ 0_{k,d}\end{bmatrix}^{\sfT}
\begin{bmatrix}-A & R^{\sfT}\\ -R & 0 \end{bmatrix}
\begin{bmatrix}\expm((1-2\alpha)t\bA)\\ 0_{k,d}\end{bmatrix}\\
= \dot{\Delta} - (1-\alpha)\Delta e^{-(1-2\alpha)t\bA} A e^{(1-2\alpha)t\bA}=
(1-\alpha)\eta e^{t(1-\alpha)A}A - (1-\alpha)\Delta A =0
\end{gathered}$$
verifying the transport equation. 
\end{proof}    
Putting everything together, we have
\begin{theorem}For $Y\in \St{n}{d}$ and $\xi\in T_Y\St{n}{d}$, let $k\leq d$ be the rank of the column span of $\xi-YY^{\sfT}\xi$, and let $Q$ be an orthonormal basis of that span, take $Q=0$ if $\xi-YY^{\sfT}\xi=0$. Consider the decomposition in \cref{sec:isometry}
\begin{equation}  \xi = YA+QR.\end{equation}
Let $\eta\in T_Y\St{n}{d}$ be another tangent vector at $Y$. Then, the parallel transport $\Delta=\Delta(t)$ of $\eta$ along the geodesic $\gamma(t)$ with $\gamma(0) = Y, \dot{\gamma}(0) = \xi$ is given by 
    \begin{equation}\begin{gathered}
        \Delta(t) = \begin{bmatrix}Y|Q\end{bmatrix}\expm(t\begin{bmatrix}2\alpha A & -R^{\sfT}\\ R & 0\end{bmatrix})\expa(t\rP, \begin{bmatrix}Y|Q\end{bmatrix}^{\sfT}\eta) \expm(t(1-2\alpha) A) \\
        + (\eta - YY^{\sfT}\eta - QQ^{\sfT}\eta)\expm(t(1-\alpha)A)\label{eq:stieftran}.
\end{gathered}
      \end{equation}
Here, $\rP = \rP_{AR}$ is the operator on $\cW:=\AHerm_{p}\times \R^{k\times p}\subset \R^{(p+k)\times p}$ that maps $\ttW= \begin{bmatrix}\ttW_a \\ \ttW_r\end{bmatrix}\in\cW$ to $\rP(\ttW)$ as follows
\begin{equation}    
     \rP(\ttW) = \rP_{AR}(\ttW) := \begin{bmatrix}((4\alpha-1)\ttW_aA + R^{\sfT}\ttW_r)_{\asym}\\
       \alpha(\ttW_rA-R \ttW_a)\end{bmatrix}.\label{eq:rPStiefel}
\end{equation}     
\end{theorem}
\begin{proof}Decompose $\eta =  \eta_{\cS} + \eta_{\cS^o}$ with
$$\begin{gathered}\eta_{\cS} = YY^{\sfT}\eta + QQ^{\sfT}\eta,\\
  \eta_{\cS^o} = (\dI_n - YY^{\sfT})\eta - QQ^{\sfT}\eta.
\end{gathered}
  $$
We can verify $\eta_{\cS}$ is in $\cS$ and $Y^{\sfT}\eta_{\cS^o}=0$, $Q^{\sfT}\eta_{\cS^o}=Q^{\sfT}\eta - Q^{\sfT}\eta=0$, hence $\xi^{\sfT}\eta_{\cS^o}=0$. Thus, the theorem follows from \cref{lem:normal} applied to $\eta_{\cS^o}$ and the isometry in \cref{sec:isometry} applied to $\eta_{\cS}$, and \cref{sec:SOSE}. It remains to derive $\rP$ from \cref{eq:transporthor2}. With $\ttB = \begin{bmatrix} \ttB_A & -\ttB_R^{\sfT}\\\ttB_R & 0\end{bmatrix}$, $\ttV = \begin{bmatrix} \ttV_A & -\ttV_R^{\sfT}\\\ttV_R & 0\end{bmatrix}$ then
      $\cP_{\fm}\lbrack\ttB, \ttV\rbrack\dI_{d+k, d} = \begin{bmatrix} [\ttB_A, \ttV_A] - \ttB_R^{\sfT}\ttV_R + \ttV_R^{\sfT} \ttB_R \\ \ttB_R \ttV_A- \ttV_R\ttB_A \end{bmatrix}$,
      $\lbrack\ttpa\ttV, \ttB\rbrack\dI_{d+k, d} = \begin{bmatrix}[\ttV_A, \ttB_A] \\ -\ttB_R\ttV_A \end{bmatrix}$, and $[\ttpa\ttB, \ttV]\dI_{d+k, d} = \begin{bmatrix}[\ttB_A, \ttV_A] \\ -\ttV_R\ttB_A \end{bmatrix}$. Thus,     
$$\begin{gathered}
\frac{1}{2}\left( [\ttB_A, \ttV_A] +\ttV_R^{\sfT} \ttB_R  - \ttB_R^{\sfT}\ttV_R \right) - (1-2\alpha)[\ttB_A, \ttV_A] = ((4\alpha-1)\ttB_A\ttV_A + \ttV_R^{\sfT} \ttB_R)_{\asym},\\
\frac{1}{2}(\ttB_R\ttV_A-\ttV_R\ttB_A) + \frac{1-2\alpha}{2}( \ttV_R\ttB_A-\ttB_R\ttV_A )=\alpha(\ttB_R\ttV_A-\ttV_R\ttB_A)
      \end{gathered}$$
giving us the top and bottom rows of $\rP_{\ttV}\ttB$, with $A=\ttV_A, R=\ttV_R, \ttW_r = \ttB_R, \ttW_a = \ttB_A$.
\end{proof}
\begin{remark}
The transport formula can be evaluated with cost $O(C_1nd^2+C_2td^3)$ with $n > d$ because $Q, A, R$ could be computed with cost $O(nd^2)$ \cite[Corollary 2.2]{Edelman_1999}, the matrix exponential is of cost $O(d^3)$ (or more precisely $O(d^3\log_2t))$, the exponential action is of cost $O(td^3)$, and the matrix multiplications are $O(nd^2)$. Thus, the time complexity is $O(C_1nd^2+C_2td^3)$. As mentioned, for small $t$, the first $O(nd^2)$ term dominates the cost, whereas, for large $t$, the term $O(td^3)$ will dominate. We can say the algorithm is $O(nd^2)$ for fixed $t$. Since Stiefel manifolds are compact, there is an upper bound for distance between two points. Thus, there is an upper bound for reasonable times $t$.
\label{rem:cost}
\end{remark}
\begin{remark}In \cref{theo:hornatural}, the naturally reductive structure does not come directly from $\SO(n)/\SO(n-d)$, but from $\trG/\trK$ with $\trG = \SO(n)\times \SO(d)$ and $\trK= \{(\diag(H, K), H)|H\in\SO(d), K\in \SO(n-d)\}$. We denote an element $\begin{bmatrix}A & -B^{\sfT}\\B & 0\end{bmatrix}$ of $\fm$ by $\ttM = \llceil A, B\rrceil$ with $A\in\oo(d), B\in\R^{(n-d)\times d}$, then $\widetilde{\ttM} = (\llceil2\alpha A, B\rrceil, (2\alpha - 1)A)\in \tfmfq$ is the lift of $\ttM$. Set
$\ttM_i = \llceil A_i, B_i\rrceil$ for $i=1,2$ then $\cP_{\tfmfq} [\widetilde{\ttM}_1, \widetilde{\ttM}_2]$ is equal to
$$\begin{gathered} \cP_{\tfmfq}(\llceil 4\alpha^2[A_1, A_2] -B_1^{\sfT}B_2 + B_2^{\sfT}B_1, 2\alpha(B_1A_2 - B_2A_1) \rrceil, (2\alpha-1)^2[A_1, A_2] ) \\
= \mathsf{lift}\llceil (4\alpha-1)[A_1, A_2] -B_1^{\sfT}B_2 + B_2^{\sfT}B_1, 2\alpha(B_1A_2 - B_2A_1) \rrceil = \mathsf{lift}\; \mathsf{AS}(\ttM_1, \ttM_2)
\end{gathered}$$
where we simplify $4\alpha^2-(2\alpha-1)^2=4\alpha-1$, and $\mathsf{AS}(\ttM_1, \ttM_2) \in \fm$ is defined by the condition above. We can verify \cref{eq:NatReductiveInv} directly for the metric $\alpha \Tr A^{\sfT}A + \Tr B^{\sfT}B$ from 
$$\langle \ttM_1, \mathsf{AS}(\ttM_2, \ttM_3) \rangle_{\vbeta} =  \alpha\Tr( (1-4\alpha) A_1[A_2, A_3]  + 2A_1B_2^{\sfT}B_3 + 2A_2B_3^TB_1 + 2 A_3B_1^{\sfT}B_2)
$$
using $\Tr A_1B_2^{\sfT}B_3 = \Tr (A_1B_2^{\sfT}B_3)^{\sfT} = - \Tr A_1B_3^{\sfT}B_2$ and similar expressions. The right-hand side is cyclic invariant, implying \cref{eq:NatReductiveInv}. See also \cite[Theorem 2.1]{Tricerri}.
\end{remark}
\section{Flag manifolds with the canonical metric}\label{sec:flag}
Flag manifolds have found applications in data science in recent years, see \cite{YeLim,Szwagier,Pennec,Birdal} for surveys. Additional background for flag manifolds could be found in \cite{YeLim,NguyenOperator,NguyenGeodesic}. As we will see, the naturally reductive space construction in \cref{sec:submerse} works for the original bi-invariant (canonical) metric. We provide here a summary relevant to the computation.

With $n > d$, for a matrix $Y\in \St{n}{d}$, consider a partition of columns of $Y$ to $p$ blocks of sizes $(d_1,\cdots, d_p)$, $Y=\begin{bmatrix}Y_1|\cdots |Y_p\end{bmatrix}$, $Y_i\in \R^{n\times d_i}$. Consider the block diagonal group $\rK_{St} = \OO(d_1)\times \cdots \times \OO(d_p)\subset \OO(d)$, whose element $U=\diag(U_1,\cdots U_p)$ acts on $Y$ by right multiplication, $YU = \begin{bmatrix}Y_1U_1|\cdots |Y_pU_p\end{bmatrix}$.

A flag manifold is the quotient of $\St{n}{d}$ by this action \cite[section 4.2]{YeLim}. It is also typically expressed as a quotient of $\SO(n)$, allowing us to leverage our framework. Set $d_{p+1} = n-d$, the flag manifold $\St{n}{d}/\rK_{St}$ will be denoted by $\Flag(\vd)$, where $\vd=(d_1,\cdots, d_p,d_{p+1})$. We can also consider $\Flag(\vd)$ as a quotient of $\SO(n)$ by the block diagonal subgroup $\mathrm{S}(\OO(d_1)\times \cdots \times \OO(d_p)\times \OO(d_{p+1}))$ \cite{YeLim} ($\mathrm{S}$ denotes the subgroup of determinant $1$). 

Going forward, we will use the term {\it flag diagonal blocks} for this block structure. Set $\rG=\SO(n), \fg=\oo(n)$ and $\fk= \oo(d_1)\times \cdots \times \oo(d_p)\times \oo(d_{p+1})$. Hence, $\fk$ is a block diagonal Lie algebra of antisymmetric matrices, while its complement $\fm$ is the space of antisymmetric matrices with zero flag diagonal blocks. Any nontrivial block subalgebra $\fa$ would overlap  $\fk$, thus, to use the result of \cref{sec:submerse}, we will take $\fa=\{0\}$ and consider the naturally reductive manifold $\trG/\trK$ with $\trG=\rG$ and $\trK=\rK$. Note that this metric corresponds to $\alpha = \frac{1}{2}$ in \cref{sec:SOSE,sec:stiefel}, with $\beta_0= -\frac{1}{2}$ and $\rB=\frac{1}{2}\dI_{\fa}$.

To compute in the more efficient Stiefel coordinates, we proceed as before. In Stiefel coordinates, the Levi-Civita connection is given by \cite[eq. 6.4]{NguyenOperator}.
\begin{equation}
  \label{eq:flag_connect}
  \nabla^{\cH}_{\ttZ}\ttW= \rD_{\ttZ}\ttW +\bY(\ttZ^{\sfT}\ttW)_{\fsym} + \frac{1}{2}(\dI_n - YY^{\sfT})(\ttZ\ttW^{\sfT}+\ttW\ttZ^{\sfT})\bY.
\end{equation}  
Here, we use $\alpha=\frac{1}{2}$ and for a matrix $A\in \R^{d\times d}$, $A_{\fsym} = \frac{1}{2}(A + A^{\sfT}) + \frac{1}{2}\cP_{\fk}(A-A^{\sfT})$, the projection to $\fk$ is the operation extracting the diagonal blocks corresponding to $\oo(d_i)$, $i=1\cdots p$. Thus, we symmetrize $A$ but add back the antisymmetric parts of the flag diagonal blocks, hence $A_{\fsym}$ means symmetrizing $A$ but leaving the flag diagonal blocks unchanged. This accounts for the difference between the full tangent space of a Stiefel manifold and the horizontal space of a flag manifold (see the extra $\cP_{\fk}$ term in \cref{eq:maingammahor}). This is proved in \cite{NguyenOperator}, but we can also derive it from \cref{eq:maingammahor}.


Given an initial horizontal vector $\xi$ in $T_Y\St{n}{p}$, we can form the vector subspace $\cS\subset \R^{n\times d}$ with columns spanned by columns of $Y$ and $\xi$. Decompose $\xi = YA +QR$ with $Q\in \R^{n\times k}$, then $\cM_{\cS}$ is isometric via the map $V\mapsto F_{YQ}V$ to $\St{d+k}{k}$, where $F_{YQ} = \begin{bmatrix}Y|Q\end{bmatrix}^{\sfT}$ as before. The group $\rK_{St} = \OO(d_1)\times \cdots \times \OO(d_p)$ operates on $\cM_{\cS}$ and on $\St{d+k}{k}$. By associativity $F_{YQ}(VU) = (F_{YQ}V)U$ if $U\in \OO(d) \supset\rK_{St}$, thus horizontal spaces map to horizontal spaces. Therefore, flag geodesics and transports between these two spaces correspond, and we obtain geodesic formulas in Stiefel coordinates. Similarly, we can map parallel transports with initial horizontal vectors in $\cM_{\cS}$.

More specifically, let $\eta$ be a horizontal vector to be transported, we decompose $\eta = \eta_{\cS} + \eta_{\cS^o}$, where $\eta_{\cS^o}$ horizontal and normal to both $Y$ and $\cS$. Similar to the Stiefel case, the transport of $\eta_{\cS^o}$ is given by \cref{lem:normal}.
This is proved exactly like in the Stiefel case, using \cref{eq:flag_connect}, with $\eta_{\cS^o} = (\dI_n - YY^{\sfT})\eta - QQ^{\sfT}\eta$, then $\gamma^{\sfT}\Delta_{\cS^o} = \dot{\gamma}^{\sfT}\Delta_{\cS^o}=0$, and the rest follows by solving a linear ODE explicitly. Thus,
\begin{proposition}With the metric \cref{eq:StiefelMetric} for $\alpha=\frac{1}{2}$, at $Y\in \St{n}{p}$, let $\xi = YA + QR$ be a horizontal vector as in \cref{sec:stiefel}, the parallel transport of a horizontal vector $\eta\in T_Y\St{n}{p}$ of $\Flag(\vd)$ lifts to $\St{n}{p}$ by 
      \begin{equation}\begin{gathered}
        \Delta(t) = \begin{bmatrix}Y|Q\end{bmatrix}\expm(t\begin{bmatrix} A & -R^{\sfT}\\ R & 0\end{bmatrix})\expa(t\rP, \begin{bmatrix}Y|Q\end{bmatrix}^{\sfT}\eta)  \\
        + (\eta - YY^{\sfT}\eta - QQ^{\sfT}\eta)\expm(\frac{t}{2}A)\label{eq:flagtrans}.
\end{gathered}
      \end{equation}
With $\ttV = \begin{bmatrix} A &-R^{\sfT}\\R & 0\end{bmatrix}$ and $\ttW=\begin{bmatrix} \ttW_a & - \ttW_r^{\sfT}\\ \ttW_r & 0\end{bmatrix}$ then $\rP$ is given by
  \begin{equation}    
     \rP(\ttW) = \rP_{\ttV}(\ttW) := \begin{bmatrix}\cP_{\fm}(\ttW_aA + R^{\sfT}\ttW_r)_{\asym}\\
       \frac{1}{2}(\ttW_rA-R \ttW_a)\end{bmatrix}.\label{eq:rPFlag}
  \end{equation}      
\end{proposition}
The proof is by substituting in \cref{eq:transporthor,eq:transporthor2} then take the first $d$ columns as in the proof of \cref{eq:rPStiefel}. The operator $\cP_{\fm}$ in \cref{eq:rPFlag} sets the flag diagonal blocks to zero. 
\begin{remark}We can equip the quotient $\St{n}{d}/\rK_{St}$ with the deformation metric on $\St{n}{d}$ for $\alpha \neq \frac{1}{2}$, giving the flag manifold a different metric. It can be shown by substitution that in this case, while the Levi-Civita connection and the geodesics are of similar formats, the part of the parallel transport equation corresponding to the exponential action becomes the solution of an ODE with nonconstant coefficients.    
\end{remark}
\section{Numerical studies}\label{sec:numer}We implemented the formulas in Python. The tables and graphs are produced from Google colab workbooks in the folder \href{https://github.com/dnguyend/par-trans/tree/main/examples}{{\it examples}} in \cite{NguyenSomParTrans} on free sessions with two Xenon CPUs and $13$GB of RAM.

In \cref{alg:Parallel}, we describe the parallel transport algorithm for Stiefel and flag manifolds (with $\alpha=\frac{1}{2}$ for flag manifolds). In the decomposition $\xi = YA + QR$, $A = Y^{\sfT}\xi$, while $Q$ and $R$ could be computed by the SVD or QR decomposition of $\xi - YY^{\sfT}\xi$.
We call \cref{alg:expv} to calculate $W$ in the parallel transport. As indicated in \cref{sec:expact}, instead of computing $\expa(\rP_{AR, \cM}, B)$ directly (the operator denotes the map $\rP_{AR}$ in either \cref{eq:rPStiefel} or \cref{eq:rPFlag} based on $\cM$), we define the function $\sfS_c: \begin{bmatrix}T\\ B\end{bmatrix} \mapsto \begin{bmatrix}c T\\B\end{bmatrix}$ scaling the top block of a matrix, then compute the exponential of $\rP^{bal}_{AR, \cM} = \sfS_{\alpha^{\frac{1}{2}}}\circ\rP_{AR, \cM}\circ\sfS_{\alpha^{-\frac{1}{2}}}$, which is antisymmetric in the Frobenius norm (see \cref{lem:customP}), and has better convergence behavior. The number of Taylor terms ($m^*$) and the scaling factor $(s^*)$ are based on \cite{HighamExpAction}, as discussed in \cref{sec:expact}.
\begin{algorithm}
\caption{Parallel transport algorithm for Stiefel and flag manifolds}
\label{alg:Parallel}
\begin{algorithmic}[1]
\Require $Y, \xi$ and the manifold $\cM$.
\State{$A, Q, R \gets YQ\_basis(Y, \xi)$}\;\hfill{Express $\xi = YA+QR$ as in \cref{lem:isometry}}
\State $norm\_est \gets calc\_norm\_est(A, R)$;\hfill{The estimate in \cref{lem:customP}}
\State{$m^*, s^* \gets lookup\_m\_s(t\times norm\_est)$}\;\hfill{Based on table 3.1 in \cite{HighamExpAction}}

\State{$W\gets \mathsf{S}_{\alpha^{-\frac{1}{2}}}\expa(\rP^{bal}_{AR,\cM}, \sfS_{\alpha^{\frac{1}{2}}}\begin{bmatrix} Y^{\sfT}\eta\\ Q^{\sfT}\eta \end{bmatrix}, t, m^*, s^*)$}\hfill{Using \cref{alg:expv}}
\State{\Return{$\begin{gathered}\lbrack Y|Q\rbrack\expm(t\begin{bmatrix} 2\alpha A & -R^{\sfT} \\ R & 0\end{bmatrix}) W e^{t(1-2\alpha)A}
+ (\eta - YY^{\sfT}\eta - QQ^{\sfT}\eta)e^{t(1-\alpha)A}.\end{gathered}$}}
\end{algorithmic}
\end{algorithm}    

In \cref{tab:byn}, we show the computational time along a time grid, for different $n$ and $\alpha\in \{\frac{1}{2}, 1\}$ for Stiefel manifolds, and for $\alpha=\frac{1}{2}$ for flag manifolds. The table generally shows sublinear growths in execution time in $n$ and $t$, with large $t$ behavior closer to linear. This is consistent with the $O(C_1nd^2 + C_2td^3)$ cost. \cref{tab:byd} shows execution time by $d$, also along the same grid. The cost is sub-cubical in $d$, and is closer to cubical as $t$ increases, again, consistent with a sum of $d^2$ and $d^3$ terms.

In \cref{fig:stiefel_n_d}, we show the approximate $O(nd^2)$ cost for small $t$ and larger range of $d$ and $n$ for \href{https://github.com/dnguyend/par-trans/blob/main/examples/NumpyStiefelParallel.ipynb}{Stiefel manifolds} (the situation for \href{https://github.com/dnguyend/par-trans/blob/main/examples/NumpyFlagParallel.ipynb}{flag manifolds} is similar), with $d$ up to $500$ and $n$ up to $16000$. The graph shows a sub-quadratic cost (without the exponential action component) in $d$ (values at $d = 100, 200, 400$ are	$0.100, 0.283$, and $0.882$). The cost in $n$ is also sublinear (values at $n=1000, 2000, 4000$, $8000$ and $16000$ are $0.0162, 0.0203, 0.0307, 0.0458$, and $0.0802$, respectively). The fact that we do not get exactly quadratic and linear costs is likely due to fixed time cost incurred in the computer operation.

\begin{figure}
  \centering
\begin{subfigure}{.5\textwidth}
  \centering
  \includegraphics[width=.9\linewidth]{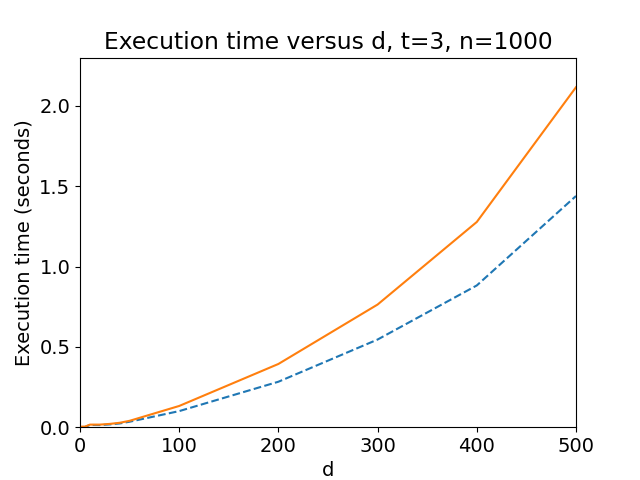}
  \caption{Time versus $d$}
  \label{fig:stief_d}
\end{subfigure}%
\begin{subfigure}{.5\textwidth}
  \centering
  \includegraphics[width=.9\linewidth]{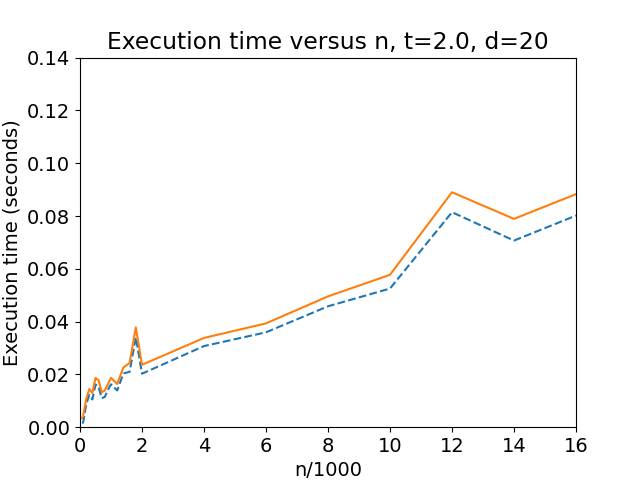}
  \caption{Time versus $n$}
  \label{fig:stief_n}
\end{subfigure}%
\caption{Execution time versus $n$ and $d$ for Stiefel manifolds. The dotted lines are times without the cost of the exponential action. The solid lines are with the exponential action.
On the left, we fix $n=1000$ and let $d$ vary from $2$ to $500$. The graph is approximately quadratic in $d$.
On the right, we set $d=20$ and let $n$ vary from $70$ to $16000$. The graph has a linear trend in $n$. Note the gap due to the exponential action increases with $d$ but is relatively constant versus $n$, consistent with the cost $O(td^3)$.}
\label{fig:stiefel_n_d}  
\end{figure}

Using automatic differentiation, in \cite{NguyenSomParTrans} we can verify that these correspond to highly accurate computation, depending on the manifold, with the transport equation satisfied with absolute accuracy around $10^{-10}$ to $10^{-12}$.
\begin{table}
  \small
\begin{tabular}{l|l|rrrrr|rrrrr}
\toprule
&  & \multicolumn{5}{c|}{Stiefel man. comp. time } & \multicolumn{5}{c}{Flag man. comp. time.} \\
\midrule
&  & \multicolumn{5}{c|}{Geodesic time grid $t$ } & \multicolumn{5}{c}{Geodesic time grid $t$} \\
$\alpha$ & n & 0.5 & 1.0 & 2.0 & 5.0 & 20.0 & 0.5 & 1.0 & 2.0 & 5.0 & 20.0 \\
\midrule
\multirow[t]{4}{*}{$\frac{1}{2}$} & 100 & 0.05 & 0.10 & 0.15 & 0.25 & 0.69 &  0.06 & 0.08 & 0.15 & 0.22 & 0.57  \\
 & 200 &  0.08 & 0.12 & 0.17 & 0.32 & 0.80  &    0.09 & 0.13 & 0.17 & 0.30 & 0.75 \\
 & 1000 & 0.15 & 0.18 & 0.26 & 0.47 & 1.45 &   0.14 & 0.18 & 0.26 & 0.49 & 1.41 \\
& 2000 & 0.18 & 0.24 & 0.34 & 0.59 & 1.78 & 0.14 & 0.25 & 0.32 & 0.54 & 1.93 \\
\cline{1-12}
\multirow[t]{4}{*}{1} & 100 & 0.09 & 0.16 & 0.22 & 0.50 & 1.15 \\
 & 200  &  0.12 & 0.17 & 0.25 & 0.45 & 1.44  &&&&& \\
 & 1000 & 0.18 & 0.21 & 0.35 & 0.68 & 2.36 &&&&&\\
 & 2000&  0.21 & 0.28 & 0.41 & 0.85 & 2.88 &&&&&\\
\cline{1-12}
\bottomrule
\end{tabular}  
\caption{Computational time for parallel transport for Stiefel and flag manifolds by $n$ for $d = 50$. The flags are $(20, 15, 15, n-d)$. The canonical metric corresponds to $\alpha=\frac{1}{2}$, the Euclidean embedded metric corresponds to $\alpha=1$. The table shows sublinear growth in execution time as $n$ and $t$ increase. Execution is on a free colab session. All time units are in second(s).}
\label{tab:byn}
\end{table}

\begin{table}
  \small
\begin{tabular}{l|l|rrrrr|rrrrr}
\toprule
&  & \multicolumn{5}{c|}{Stiefel man. comp. time} & \multicolumn{5}{c}{Flag man. comp. time.} \\
\midrule
&  & \multicolumn{5}{c|}{Geodesic time grid $t$} & \multicolumn{5}{c}{Geodesic time grid $t$} \\
$\alpha$ & d & 0.5 & 1.0 & 2.0 & 5.0 & 20.0 & 0.5 & 1.0 & 2.0 & 5.0 & 20.0 \\
\midrule
\multirow[t]{4}{*}{$\frac{1}{2}$} & 5& 0.01 & 0.02 & 0.04 & 0.08 & 0.25  &&&&&\\
 & 10 & 0.03 & 0.06 & 0.10 & 0.15 & 0.43  &  0.03 & 0.05 & 0.08 & 0.15 & 0.43 \\
 & 20 & 0.05 & 0.08 & 0.13 & 0.21 & 0.53  &  0.05 & 0.07 & 0.15 & 0.21 & 0.54 \\
& 100 & 0.42 & 0.60 & 0.99 & 2.55 & 9.18 & 0.45 & 0.57 & 1.13 & 2.35 & 8.60 \\
\cline{1-12}
\multirow[t]{4}{*}{1} &5 &  0.01 & 0.02 & 0.04 & 0.10 & 0.37 &&&&&\\
 & 10 & 0.04 & 0.07 & 0.13 & 0.23 & 0.61&&&&&\\
 & 20 &  0.07 & 0.10 & 0.15 & 0.25 & 0.73  &&&&&\\
 & 100 & 0.60 & 0.92 & 1.68 & 3.92 & 15.80 &&&&&\\
\cline{1-12}
\bottomrule
\end{tabular}  
\caption{Computational time for parallel transport for Stiefel and flag manifolds by $d$ for $n = 1000$. We scale the partial flag of $(d_1, d_2, d_3)=(4, 3, 3)$ by $\frac{d}{4+3+3}$ to $d$ for flag manifolds, setting $d_4=n-d$. Computational time is subcubical in $d$, consistent with the cost $O(C_1nd^2+td^3)$. For large $t$, the trend is approximately linear in $t$. All time units are in seconds.}
\label{tab:byd}
\end{table}

\subsection{Parallel transport as isometries}We also demonstrate numerically the isometric property of parallel transport as implemented by our method, by showing the inner product matrix of a set of vectors is preserved during transportation. We consider a geodesic starting at $\dI_{2000,100}$, with twenty tangent vectors of integer lengths sampled from integers between $1$ and $60$, while the relative angles are randomly generated. The geodesic initial velocity vector $v$ has length $1$. We compute the inner product matrix of the random tangent vectors at $t=0$ and of the transported vectors along a time grid. We measure the maximum absolute difference between the inner product matrices at each grid point and the initial point and plot the $\log_{10}$ of the differences in \cref{fig:isometry}. The absolute errors are around $10^{-12}$ up to $t=15$.
\begin{figure}
  \centering
\begin{subfigure}{.5\textwidth}
  \centering
  \includegraphics[width=1.05\linewidth]{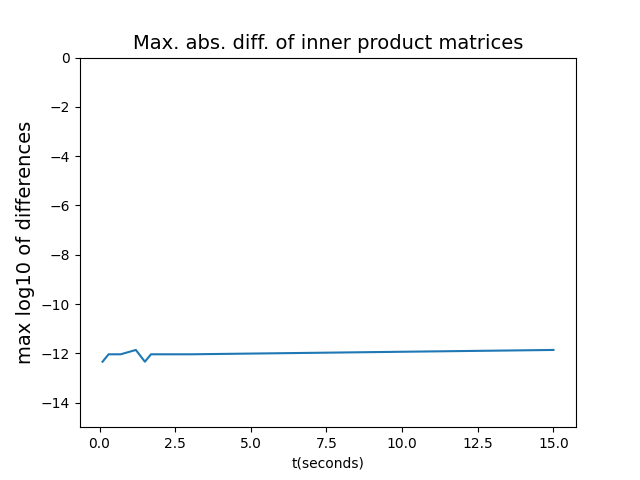}
  \caption{Stiefel manifolds}
  \label{fig:stief}
\end{subfigure}%
\begin{subfigure}{.5\textwidth}
  \centering
  \includegraphics[width=1.05\linewidth]{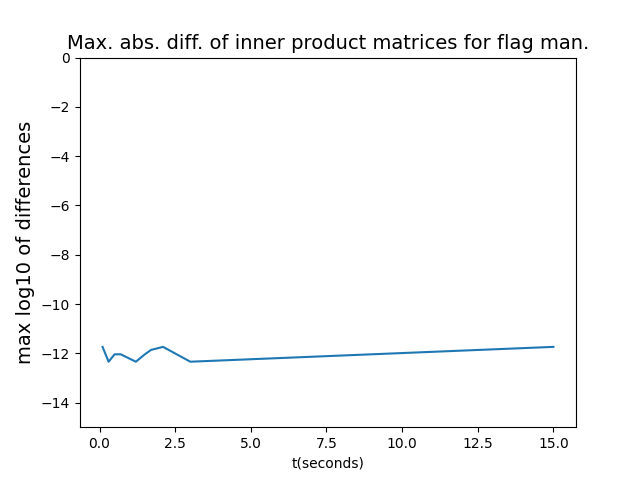}
  \caption{Flag manifolds}
  \label{fig:flag}
\end{subfigure}%
\caption{Isometry of parallel transport of Stiefel and flag manifolds. We transport a set of $20$ vectors from an initial point on $\St{2000}{100}$, with $\alpha=1$ along a geodesic, and compute the maximum difference between the inner product matrices of these $20$ vectors at $t=0$ versus $t\in \{0.1, .3, .5, .7, 1.2, 1.5, 1.7, 2.1, 3., 15.\}$. On the right-hand side, we consider $\Flag(50, 20, 30, 1900)$ with the canonical metric $\alpha=\frac{1}{2}$. Up to $t=15$ seconds, the differences are in the order of $10^{-12}$.}
\label{fig:isometry}  
\end{figure}
\section{Discussion}In this paper, we provide an efficient algorithm to compute parallel transport on Stiefel and flag manifolds, as well as transport formulas for the general linear group and the special orthogonal group. We also provide a framework for parallel transports on groups with a deformation metric based on a pair of Lie algebras with bi-invariant metrics, in particular, a pair of transposable Lie algebras. We believe the framework developed here will be useful for other manifolds in applied mathematics.
\section*{Acknowledgments}
We wish to thank the two anonymous reviewers and the editor for the careful review and numerous suggestions that helped us improve the
article significantly. In particular, we are grateful for one reviewer’s strong suggestion to use the reductive space approach, which pointed us in the direction of establishing the naturally reductive property. The new approach gives us a much clearer perspective of the transport formula, obtained by solving the transport equation directly in our first version. The same reviewer’s suggestion also leads to a better derivation of the Christoffel functions from the geodesic equations. We thank him\slash her for the references \cite{Rentmee,MataigneZimmermannMiolane} in relation to \cref{lem:isometry} as well as the observation that the algorithm could be considered $O(nd^2)$ for bounded time. Needless to say, any remaining mistake is our sole responsibility.

Du Nguyen would like to thank his family for their loving support in this project. He thanks Mr. Tom Szwagier for feedback on an earlier version of the article. The work was supported by a research grant (VIL40582) from VILLUM FONDEN, and the Novo Nordisk Foundation grants NNF24OC0093490 and NNF24OC0089608.

\appendix
\section{Exponential action}\label{sec:expact}The solution of a linear ODE with constant coefficient $\dot{X} = AX$ and initial condition $X(0) = X_0$ is given by $\expm(tA)X_0$, where typically $X$ is a vector and $A$ is a matrix, where $\expm$ is the matrix exponential. We often need to treat the case where $X_0$ is a matrix and $A$ is an operator on a matrix space. In that situation, which is the situation encountered in this article, computing the exponential by vectorization is usually inefficient. The better alternative is to compute the {\it exponential action}, denoted  by $\expa(tA, X_0)$. Instead of evaluating the full matrix exponential, we approximate the action $\expm(tA)X_0$ more efficiently, often \cite{HighamExpAction} by evaluating a {\it scaling and squaring} of a Taylor expansion of the operator exponential operating on $X_0$.

Specifically, the popular library algorithm for exponential action in Python (function $expm{\_}multiply$ in the scipy package \cite{scipy-nmeth} and similar functions for other languages, $expmv$ in MATLAB and $expv$ in Julia) are based on the work of Al-Mohy and Higham \cite{HighamExpAction}, approximating $e^A B= (e^{1/sA})^s B \approx (T_m(s^{-1} A))^sB$, where $T_m$ denotes the $m$ terms Taylor series expansion of $e^A$. The cited paper provides an algorithm to find a pair of optimal $(m, s)$, based on estimates of $\|A^p\|^{1/p}_1$ for several values of $p$. The paper found the higher $p$ estimates are beneficial when $A$ is not normal, for example, when $A$ is upper triangular. These estimates are used to determine the Taylor order $m$.

The cost is thus a multiple of the cost of evaluating $Av$ for a vector (or matrix) $v$ of the same size with $X_0$. The multiplier is dependent on the required accuracy and $|t|\|A\|$, for an appropriate norm $\|A\|$. In library implementations, $A$ could be given as a function. To estimate the norm, we also need to evaluate the transpose $A^{\sfT}$.

In these implementations, the $1$-norm is estimated using the algorithm in \cite{HighamTisseur}. We will use $expm{\_}multiply$ from scipy for $\GL(n)$ and $\SO(n)$, while we will show a customized implementation based on efficient estimates of the $1$-norm of the operator $\rP$ for Stiefel and flag manifolds in \cref{lem:customP} could speed up the algorithm significantly, without the higher $p$ estimates. In particular, we show for these manifolds, after rescaling, the operator $\rP$ becomes antisymmetric in the Frobenius norm. Numerically, in this case, the estimate using just the 1-norm $\|\rP\|_1$ is adequate.

Our next observation is for our problem, $\|\rP\|_1$ can be estimated more efficiently than by calling a library function. Recall that on a $n$-dimension vector space with a basis $\{e_i\}_{i=1}^n$, the 1-norm of a vector $\sum c_i e_i$ is $\sum_i|c_i|$. Vectorizing an operator $\rP$ using this basis, the operator $1$-norm is given by $\|\rP\|_1= \max_j |\rP e_j|_1$, which has a cost of $O(n^2)$, thus, it is $O(m^2d^2)$ on a matrix space $\R^{m\times d}$. The sampling in \cite{HighamTisseur} reduces this cost, but we can reduce it further by exploiting the structure of $\rP$.
For the Stiefel and flag manifold cases, we provide an estimate based on norms of $A$ and $R$. The resulting customized implementation of exponential action could be more than $10$ times faster than the library version, significantly reducing the total time cost of the parallel transport calculation. To keep focus, we do not attempt these estimates of $\|\rP\|_1$ for $\GL^+(n)$ and $\SO(n)$ (the latter uses a similar estimate to Stiefel manifolds). Likewise, we will not discuss the scalings to make $\rP$ antisymmetric in the Frobenius norm for these manifolds as in \cref{lem:customP} below.

Further, the scipy function $expm{\_}multiply$ is not yet implemented in JAX \cite{jax2018github}, a Python library that provides easy compilation to run in GPU and allows automatic differentiation (AD). We use our estimates to implement a simple version of $expv$ in JAX, and use the AD feature to compute the derivative of the parallel transport, verifying that our implementation is highly accurate. We offer $expm{\_}multiply$-based parallel transport for all four manifolds while providing customized implementations for Stiefel and flag manifolds, with AD(JAX) and without-AD versions. The discussion here will likely apply to other languages using similar algorithms. 

First, for a real number $c$, define $\sfS_{c}$ to be the operator on $\R^{m\times d}$ with $m \geq d$, scaling the top $d\times d$ block by $c$ and preserving the bottom $(m-d)\times d$ block. We have $\expm \rP = \sfS_c^{-1}\expm (\sfS_c\rP \sfS_c^{-1})\sfS_c$ for an operator $\rP$ on $\R^{m\times d}$. The following computes an upper bound for $\|\sfS_{\alpha^{\frac{1}{2}}}\rP_{AR} \sfS_{\alpha^{-\frac{1}{2}}}\|_1$ with a cost of $O(d^2)$:
\begin{lemma}\label{lem:customP}
  For $k\geq 0$ and $A=(a_{ij})_{ij=(1,1)}^{(d,d)}\in\oo(d), R=(r_{ij})_{ij=(1,1)}^{(k,d)}\in\R^{k\times d}$, with respect to the standard basis of $\R^{(d+k)\times d}$, the operator $\rP_{AR}^{bal} := \sfS_{\alpha^{\frac{1}{2}}}\circ \rP_{AR}\circ \sfS_{\alpha^{-\frac{1}{2}}}$ with
  $\rP_{AR}$ in \cref{eq:rPStiefel} operating on $\R^{(d+k)\times d}$ has an upper bound on the $1$-norm
  \begin{equation}\begin{gathered}
    \|\rP_{AR}^{bal}\|_1 = \max_{|\ttW|_1=1} \lvert \rP_{AR}^{bal}\ttW \rvert_1 \leq
    \max(n_A, n_R)\\
    \text{where }
    n_A = \max_j(\alpha^{\frac{1}{2}}\sum_{i=1}^d(\lvert r_{ij}\rvert
                          + \lvert 4\alpha-1\rvert\|A\|_1)),\\
                          n_R = \max_j(\alpha\sum_i(\lvert a_{ij}\rvert + \alpha^{\frac{1}{2}}\|R\|_{\infty})).
    \end{gathered}    
  \end{equation}
Also, define the operator $\mathsf{skew}_{top}$ on $\R^{(d+k)\times d}$ to be the map antisymmetrizing the top $d\times d$ block, then
$\rP_{AR}^{bal}\circ \mathsf{skew}_{top}: \begin{bmatrix}\ttW_a\\\ttW_r\end{bmatrix}\mapsto \rP_{AR}^{bal}\begin{bmatrix}(\ttW_a)_{\asym}\\ \ttW_r \end{bmatrix}$  is antisymmetric in the Frobenius inner product. 
With $\rP^{\Flag}_{AR}$ given in \cref{eq:rPFlag} at $\alpha=\frac{1}{2}$, we have the same bound
for $\sfS_{\alpha^{\frac{1}{2}}}\circ \rP^{\Flag}_{AR}\circ \sfS_{\alpha^{-\frac{1}{2}}}$.  
\end{lemma}
\begin{proof}For $\ttW=\begin{bmatrix}\ttW_a\\ \ttW_r\end{bmatrix}$, we have
$$\rP_{AR}^{bal}\ttW
=\begin{bmatrix}\alpha^{\frac{1}{2}}((4\alpha-1)\alpha^{-\frac{1}{2}}\ttW_aA + R^{\sfT}\ttW_r)_{\asym}\\
       \alpha (\ttW_rA-\alpha^{-\frac{1}{2}}R \ttW_a)\end{bmatrix}
      = \begin{bmatrix}((4\alpha-1)\ttW_aA + \alpha^{\frac{1}{2}}R^{\sfT}\ttW_r)_{\asym}\\
        \alpha \ttW_rA-\alpha^{\frac{1}{2}}R \ttW_a\end{bmatrix}.$$
Consider the basis vectors $E_{ij}$ with $1\leq i\leq d+k,1\leq j\leq d$, where $E_{ij}$ is an elementary matrix with the $(i,j)$-entry is one and other entries vanish. Write $A = \sum_{qs=(1,1)}^{(d,d)}a_{qs}E_{qs}, R= \sum_{qs=(1,1)}^{(k,d)}r_{qs}E_{(d+q)s}$. For $i\leq d$, 
 $$\begin{gathered}
   |\rP^{bal}_{AR}E_{ij}|_{\vec{1}} =\bigg\lvert \begin{bmatrix}\frac{4\alpha-1}{2}(\sum_{r=1}^d a_{jr}E_{ij}E_{jr} - \sum_{r=1}^d a_{jr} E_{rj}E_{ji}) \\ -\alpha^{\frac{1}{2}} \sum_{q=1}^k r_{qi}E_{(d+q)i}E_{ij} \end{bmatrix}\bigg\rvert_{\vec{1}}\\
   \leq \frac{|4\alpha-1|}{2}(\sum_{r=1}^d |a_{jr}| + \sum_{s=1}^d |a_{js}|) +
   \alpha^{\frac{1}{2}} \sum_{q=1}^k |r_{qi}|\leq n_A
\end{gathered}
$$
where $|.|_{\vec{1}}$ denotes the $1$-norm after vectorization, because $E_{ij}E_{jr} = E_{ir}$ is distinct from $E_{sj}E_{ji}=E_{si}$ unless $i=r=s$, in this case, we have a cancellation. For $i>d$
  $$\begin{gathered}
   |\rP^{bal}_{AR}E_{ij}|_{\vec{1}} =\bigg\lvert \begin{bmatrix}\alpha^{\frac{1}{2}}(\sum_{s=1}^dr_{i-d,s}E_{si}E_{ij})_{\asym}\\
     \alpha E_{ij}\sum_s a_{js}E_{js}\end{bmatrix}\bigg\rvert_{\vec{1}}\leq \frac{2\alpha^{\frac{1}{2}}}{2}\sum_s |r_{i-d,s}| +\alpha \sum_s |a_{js}|\leq n_R
   .
   \end{gathered}
 $$
This gives us the bound. To see $\rP_{AR}^{bal}\circ \mathsf{skew}_{top}$ is antisymmetric, note the diagonal parts $\ttW_a\mapsto (4\alpha-1)((\ttW_a)_{\asym}A)_{\asym}$ and $\ttW_r\mapsto \alpha \ttW_rA$ are both antisymmetric in the Frobenius norm, since $A$ is antisymmetric. We can verify directly the off-diagonal operators $F_1: \ttW_r\mapsto \alpha^{\frac{1}{2}}(R^{\sfT}\ttW_r)_{\asym}$ and $F_2: \ttW_a\mapsto -\alpha^{\frac{1}{2}}R(\ttW_a)_{\asym}$ satisfy $F_1=-F_2^{\sfT}$ (the operator transpose is also in the Frobenius norm). For flag manifolds, clearing flag diagonal blocks does not increase norms, thus the same estimate holds.
\end{proof}
With this estimate, following scipy's implementation, we have a finer extension of \cite[table 3.1]{HighamExpAction} containing a look up table for constants $\theta_m$ from $m$. The number of Taylor terms $m_*$ will be the value of $m$ minimizing $\mathrm{ceil}(\frac{m\|\rP\|_1}{\theta_m})$ and the scaling parameter $s$ is set to be $\mathrm{ceil}(\frac{\|\rP\|_1}{\theta_{m_*}})$, both are used in the approximation $\expa(\rP, \ttB) \approx (T_m(s^{-1} \rP))^s\ttB$.
In \cref{alg:expv} We implement an evaluation of $\expa$, given an estimate of the number of Taylor  terms and the scaling factor using the 1-norm estimate, $\rP$ is assumed to be an antisymmetric operator in the Frobenius norm (if this does not hold, convergence may be slower than the library version which uses higher norms.)
\begin{algorithm}
\caption{Computing exponential action $e^{t\rP}B$ based on the scaling method. The number of Taylor terms $m_*$ and the scaling factor $s_*$ are pre-estimated.}
\label{alg:expv}
\begin{algorithmic}[1]
\State Input: $\rP, B, t, m_*, s_*$
\State Choose a relative tolerance level $tol$
\For{$i = 0,\cdots, s_*-1$}
\State $old\_norm \gets |B|_2$
\State $F\gets B$
\For{$j = 0,\cdots, m_*-1$}
    \State $B \gets \frac{t}{s_*(j+1)}\rP(B)$
    \State $new\_norm = |B|_2$
    \State $F \gets F + B$
    \If{$old\_norm + new\_norm < tol|F|_2$} 
    \State break
    \EndIf
    \State $old\_norm = new\_norm$
\EndFor
\EndFor
\end{algorithmic}
\end{algorithm}

If we use the library \emph{expm\_multiply} function, we need to supply the transpose operator for $\rP_{\ttV}$ in \cref{eq:transportCheeger}. We have
\begin{lemma}If $\fa$ is a transposable subalgebra and $\rB = \beta\beta_0\dI_{\fa}$, in the Frobenius inner product, for $\ttC\in \cE=\R^{n\times n}$, the adjoint of the operator $\rP_{\ttV}$ in \cref{eq:transportCheeger} is given by
\begin{equation}\rP_{\ttV}^{\sfT}\ttC = \frac{1}{2}\left([\ttC,\ttV^{\sfT}]
            + (1+\beta)([(\ttpa\ttV )^{\sfT},\ttC]
            - \ttpa[\ttC,\ttV^{\sfT}])\right).\label{eq:padj}
\end{equation}
\end{lemma}
\begin{proof}We use the equality $\Tr \ttB_1^{\sfT} [\ttB_2, \ttX] = \Tr \ttB_2^{\sfT}[\ttB_1,\ttX^{\sfT}]$ for $\ttB_1, \ttB_2, \ttX\in\R^{n\times n}$ by taking transpose then use \cref{eq:cyclicInvariant}. Thus, for $\ttB\in \cE=\R^{n\times n}$,
  $$\begin{gathered}\Tr \ttC^{\sfT}\left( \frac{1}{2}[\ttB, \ttV] +\frac{1+\beta}{2}\left( - [\ttpa\ttB,\ttV] -[\ttB, \ttpa\ttV]   \right)\right)\\
    =\frac{1}{2}\Tr\ttB^{\sfT}[\ttC,\ttV^{\sfT}] -\frac{1+\beta}{2}\left( \Tr(\ttpa\ttB)^{\sfT}[\ttC,\ttV^{\sfT}] + \Tr\ttB^{\sfT}[\ttC, (\ttpa\ttV)^{\sfT}]\right).
    \end{gathered}
  $$
  Using $\Tr(\ttpa\ttB)^{\sfT}[\ttC,\ttV^{\sfT}] = \Tr\ttB^{\sfT}(\ttpa[\ttC,\ttV^{\sfT}])$, we get the expression for $\rP^{\sfT}_{\ttV}$.
\end{proof}

\section{The Grassmann case}\label{app:appGrass}The Grassmann manifold is a flag manifold with $p = 1$. In this case, with $\vd= (d, n-d)$, $\alpha = \frac{1}{2}$, we have $\rP$ is identically zero, $A=0$, $Y^{\sfT}\eta=0$, thus, $\begin{bmatrix}Y|Q\end{bmatrix}^{\sfT}QQ^{\sfT}\eta = \begin{bmatrix}0\\ Q^{\sfT}\end{bmatrix}\eta$. We have
$$        \Delta(t) = \begin{bmatrix}Y|Q\end{bmatrix}\expm(t\begin{bmatrix} 0 & -R^{\sfT}\\ R & 0\end{bmatrix})\begin{bmatrix}0\\ Q^{\sfT}\end{bmatrix}\eta        + \eta - QQ^{\sfT}\eta.$$
If $\eta = Q\Sigma V^{\sfT} = QR$ is a compact SVD decomposition, where $R:=\Sigma V^{\sfT}$, with $\Sigma$ is diagonal in $\R^{k\times k}$, $V\in \R^{k\times d}$ with $V^{\sfT}V = \dI_d$. Set $C = \begin{bmatrix} V & 0 \\ 0 & \dI_d\end{bmatrix}\in \R^{(d+k)\times (d+k)}$, then
  $$\begin{gathered}\begin{bmatrix} 0 & -R^{\sfT}\\ R & 0\end{bmatrix} = C \begin{bmatrix} 0 & -\Sigma\\ \Sigma & 0\end{bmatrix} C^{\sfT},\\
        \expm (t\begin{bmatrix} 0 & -R^{\sfT}\\ R & 0\end{bmatrix}) = \begin{bmatrix}\dI_d - VV^{\sfT} & 0\\ 0 & 0\end{bmatrix} +C \begin{bmatrix} \cos(t\Sigma) & -\sin(t\Sigma)\\ \sin(t\Sigma) & \cos(t\Sigma)\end{bmatrix} C^{\sfT},
  \end{gathered}$$
  where we use $\rC^{\sfT}\rC = \dI_{d+k}$ to simplify the middle terms $\frac{t^j}{j!}(C\begin{bmatrix} 0 & -\Sigma\\ \Sigma & 0\end{bmatrix}C^{\sfT})^j$, $j>0$ in the power series expansion of $\expm$, while the constant term has the adjustment with the term $\dI_d - VV^{\sfT}$ above. Thus, $\expm (t\begin{bmatrix} 0 & -R^{\sfT}\\ R & 0\end{bmatrix})\begin{bmatrix}0\\ Q^{\sfT}\end{bmatrix}\eta$ simplifies to $C\begin{bmatrix} -\sin(t\Sigma)\\ \cos(t\Sigma)\end{bmatrix}Q^{\sfT}\eta$, and 
  \begin{equation} \Delta(t) = \begin{bmatrix}YV|Q\end{bmatrix}\begin{bmatrix} -\sin(t\Sigma)\\ \cos(t\Sigma)\end{bmatrix}Q^{\sfT}\eta  + \eta - QQ^{\sfT}\eta,\label{eq:grasstrans}
    \end{equation}
  which is eq (2.68) in \cite{Edelman_1999}.

\bibliographystyle{amsplain}
\bibliography{stiefel_parallel}

\providecommand{\bysame}{\leavevmode\hbox to3em{\hrulefill}\thinspace}
\providecommand{\MR}{\relax\ifhmode\unskip\space\fi MR }
\providecommand{\MRhref}[2]{%
  \href{http://www.ams.org/mathscinet-getitem?mr=#1}{#2}
}
\providecommand{\href}[2]{#2}
\begin{thebibliography}{10}

\bibitem{AMS_book}
P.-A. Absil, R.~Mahony, and R.~Sepulchre, \emph{Optimization algorithms on matrix manifolds}, Princeton University Press, Princeton, NJ, USA, 2007.

\bibitem{HighamExpAction}
A.~H. Al-Mohy and N.~Higham, \emph{Computing the action of the matrix exponential, with an application to exponential integrators}, SIAM Journal on Scientific Computing \textbf{33} (2011), no.~2, 488--511.

\bibitem{Barp2019}
Alessandro Barp, \emph{Hamiltonian {M}onte {C}arlo on {L}ie groups and constrained mechanics on homogeneous manifolds}, Geometric Science of Information (Cham) (Frank Nielsen and Fr{\'e}d{\'e}ric Barbaresco, eds.), Springer International Publishing, 2019, pp.~665--675.

\bibitem{jax2018github}
James Bradbury, Roy Frostig, Peter Hawkins, Matthew~James Johnson, Chris Leary, Dougal Maclaurin, George Necula, Adam Paszke, Jake Vander{P}las, Skye Wanderman-{M}ilne, and Qiao Zhang, \emph{{JAX}: composable transformations of {P}ython+{N}um{P}y programs}, 2018.

\bibitem{Cheeger1973}
J.~Cheeger, \emph{{Some examples of manifolds of nonnegative curvature}}, Journal of Differential Geometry \textbf{8} (1973), no.~4, 623 -- 628.

\bibitem{DAZ}
J.E. D’Atri and W.~Ziller, \emph{Naturally reductive metrics and {E}instein metrics on compact {L}ie groups}, Memoirs of the American Mathematical Society \textbf{215} (1979).

\bibitem{Edelman_1999}
A.~Edelman, T.~A. Arias, and S.~T. Smith, \emph{The geometry of algorithms with orthogonality constraints}, SIAM J. Matrix Anal. Appl. \textbf{20} (1999), no.~2, 303--353.

\bibitem{Fiori}
Simone Fiori, \emph{Solving minimal-distance problems over the manifold of real-symplectic matrices}, SIAM Journal on Matrix Analysis and Applications \textbf{32} (2011), no.~3, 938--968.

\bibitem{Gallier}
J.~Gallier and J.~Quaintance, \emph{Differential geometry and {L}ie groups, {I}}, vol.~12, Springer, New York, NY, 2020.

\bibitem{GohbergIndefinite}
I.~Gohberg, P.~Lancaster, and L.~Rodman, \emph{Indefinite linear algebra and applications}, Birkh{\"a}user Basel, 2006.

\bibitem{Gordon1985}
C.S. Gordon, \emph{Naturally reductive homogeneous {R}iemannian manifolds}, Canadian Journal of Mathematics \textbf{37} (1985), 467--487.

\bibitem{GZ2000}
K.~Grove and W.~Ziller, \emph{Curvature and symmetry of {M}ilnor spheres}, The Annals of Mathematics \textbf{152} (2000), 331--367.

\bibitem{ladder}
N.~Guigui and X.~Pennec, \emph{Numerical accuracy of ladder schemes for parallel transport on manifolds}, Found. Comput. Math. \textbf{22} (2022), 757--–790.

\bibitem{HighamTisseur}
Nicholas~J. Higham and Fran\c{c}oise Tisseur, \emph{A block algorithm for matrix 1-norm estimation, with an application to 1-norm pseudospectra}, SIAM Journal on Matrix Analysis and Applications \textbf{21} (2000), no.~4, 1185--1201.

\bibitem{HumphreysLie}
J.~Humphreys, \emph{Introduction to {L}ie algebras and representation theory}, Graduate studies in mathematics, Springer, 1972.

\bibitem{ExtCurveStiefel}
K.~H{\"u}per, I.~Markina, and F.~{Silva Leite}, \emph{{A Lagrangian approach to extremal curves on Stiefel manifolds}}, Journal of Geometric Mechanics \textbf{13} (2021), 55--72.

\bibitem{HuperFlor}
Knut H{\"u}per and Florian Ullrich, \emph{Real {S}tiefel manifolds: An extrinsic point of view}, 2018 13th APCA International Conference on Automatic Control and Soft Computing (CONTROLO), 2018, pp.~13--18.

\bibitem{Jensen}
G.~R. Jensen, \emph{{Einstein metrics on principal fibre bundles}}, Journal of Differential Geometry \textbf{8} (1973), no.~4, 599--614.

\bibitem{VMSil}
V.~Jurdjevic, I.~Markina, and F.~Silva-Leite, \emph{Extremal curves on {S}tiefel and {G}rassmann manifolds}, Journal of Geometric Analysis \textbf{30} (2020), 3948–397.

\bibitem{KobNom}
S.~Kobayashi and K.~Nomizu, \emph{Foundations of differential geometry, volume 2}, A Wiley Publication in Applied Statistics, Wiley, 1996.

\bibitem{LoAyPenn}
M.~Lorenzi, N.~Ayache, and X.~Pennec, \emph{Schild’s ladder for the parallel transport of deformations in time series of images}, Information Processing in Medical Imaging (H.K.~Hahn G.~Sz{\'e}kely, ed.), LNCS, vol. 6801, Springer, Berlin, Heidelberg, 2011, pp.~463--–474.

\bibitem{LoPenn}
M.~Lorenzi and X.~Pennec, \emph{Efficient parallel transport of deformations in time series of images: from {S}child to pole ladder}, J. Math. Imaging. Vis. \textbf{50} (2014), 5--–17.

\bibitem{Birdal}
Nathan Mankovich, Gustau Camps-Valls, and Tolga Birdal, \emph{Fun with flags: Robust principal directions via flag manifolds}, Proceedings of the IEEE/CVF Conference on Computer Vision and Pattern Recognition (CVPR), June 2024, pp.~330--340.

\bibitem{MartinNeff}
R.~J. Martin and P.~Neff, \emph{Minimal geodesics on {GL}(n) for left-invariant, right-{O}(n)-invariant {R}iemannian metrics}, Journal of Geometric Mechanics \textbf{8} (2016), no.~3, 323--357.

\bibitem{MataigneZimmermannMiolane}
Simon Mataigne, Ralf Zimmermann, and Nina Miolane, \emph{An efficient algorithm for the riemannian logarithm on the stiefel manifold for a family of riemannian metrics}, SIAM Journal on Matrix Analysis and Applications \textbf{46} (2025), no.~2, 879--905.

\bibitem{medinaRevoy}
A.~Medina and P.~Revoy, \emph{Alg\`{e}bres de {L}ie et produit scalaire invariant}, Annales scientifiques de l’\'{E}cole Normale Sup\'{e}rieure \textbf{18} (1985), 553--–561.

\bibitem{Milnor1976}
J.~Milnor, \emph{Curvatures of left invariant metrics on {L}ie groups}, Advances in Mathematics \textbf{21} (1976), no.~3, 293--329.

\bibitem{MiolanePennec}
Nina Miolane and Xavier Pennec, \emph{Computing bi-invariant pseudo-metrics on {L}ie groups for consistent statistics}, Entropy \textbf{17} (2015), no.~4, 1850--1881.

\bibitem{NguyenGeodesic}
D.~Nguyen, \emph{Closed-form geodesics and optimization for {R}iemannian logarithms of {S}tiefel and flag manifolds}, Journal of Optimization Theory and Applications \textbf{194} (2022), 142--166.

\bibitem{NguyenLie}
\bysame, \emph{Curvatures of {S}tiefel manifolds with deformation metrics}, Journal of Lie Theory \textbf{32} (2022), 563--–600.

\bibitem{NguyenOperator}
\bysame, \emph{Operator-valued formulas for {R}iemannian gradient and {H}essian and families of tractable metrics in {R}iemannian optimization}, Journal of Optimization Theory and Applications \textbf{198} (2023), 135--164.

\bibitem{NguyenSomParTrans}
D.~Nguyen and S.~Sommer, \emph{Project par-trans}, \url{https://github.com/dnguyend/par-trans}, 2024.

\bibitem{ONeil1983}
B.~O'Neill, \emph{Semi-{R}iemannian geometry with applications to relativity}, Pure and Applied Mathematics, vol. 103, Academic Press, Inc, New York, NY, 1983.

\bibitem{Pennec}
X.~Pennec, P.~Fillard, and N.~Ayache, \emph{{A Riemannian Framework for Tensor Computing}}, Int J Comput Vision (2006), 41--66.

\bibitem{Rentmee}
Q.~Rentmeesters, \emph{Algorithms for data fitting on some common homogeneous spaces}, Ph.D. thesis, Universit{\'e} Catholique de Louvain, Louvain, Belgium, 2013.

\bibitem{RingWirth}
Wolfgang Ring and Benedikt Wirth, \emph{Optimization methods on {R}iemannian manifolds and their application to shape space}, SIAM Journal on Optimization \textbf{22} (2012), no.~2, 596--627.

\bibitem{Sagle1970}
A.~A. Sagle, \emph{{Some homogeneous {E}instein manifolds}}, Nagoya Mathematical Journal \textbf{39} (1970), no.~39, 81--106.

\bibitem{SatoRiemannianCG}
Hiroyuki Sato, \emph{Riemannian conjugate gradient methods: General framework and specific algorithms with convergence analyses}, SIAM Journal on Optimization \textbf{32} (2022), no.~4, 2690--2717.

\bibitem{SmithThesis}
S.~T. Smith, \emph{Geometric optimization methods for adaptive filtering}, Ph.D. thesis, Harvard University, Cambridge, MA, USA, 1993.

\bibitem{Szwagier}
T.~Szwagier and X.~Pennec, \emph{Rethinking the {R}iemannian logarithm on flag manifolds as an orthogonal alignment problem}, Geometric Science of Information 2023, Volume 1, Springer, 2023, pp.~375--383.

\bibitem{Tricerri}
F.~Tricerri and L.~Vanhecke, \emph{Naturally reductive homogeneous spaces and generalized {H}eisenberg groups}, Compositio Mathematica \textbf{52} (1984), no.~3, 389--408.

\bibitem{VAV}
Bart Vandereycken, P.-A. Absil, and Stefan Vandewalle, \emph{{A Riemannian geometry with complete geodesics for the set of positive semidefinite matrices of fixed rank}}, IMA Journal of Numerical Analysis \textbf{33} (2012), no.~2, 481--514.

\bibitem{scipy-nmeth}
Pauli Virtanen, Ralf Gommers, Travis~E. Oliphant, Matt Haberland, Tyler Reddy, David Cournapeau, Evgeni Burovski, Pearu Peterson, Warren Weckesser, Jonathan Bright, St{\'e}fan~J. {van der Walt}, Matthew Brett, Joshua Wilson, K.~Jarrod Millman, Nikolay Mayorov, Andrew R.~J. Nelson, Eric Jones, Robert Kern, Eric Larson, C~J Carey, {\.I}lhan Polat, Yu~Feng, Eric~W. Moore, Jake {VanderPlas}, Denis Laxalde, Josef Perktold, Robert Cimrman, Ian Henriksen, E.~A. Quintero, Charles~R. Harris, Anne~M. Archibald, Ant{\^o}nio~H. Ribeiro, Fabian Pedregosa, Paul {van Mulbregt}, and {SciPy 1.0 Contributors}, \emph{Scipy 1.0: Fundamental algorithms for scientific computing in {P}ython.}, 2020, pp.~261--272.

\bibitem{YeLim}
Ke~Ye, Ken Sze-Wai Wong, and Lek-Heng Lim, \emph{Optimization on flag manifolds}, Math. Program. \textbf{194} (2022), 621–660.

\bibitem{Zimmer}
R.~Zimmermann, \emph{A matrix-algebraic algorithm for the {Riemannian} logarithm on the {Stiefel} manifold under the canonical metric}, SIAM J. Matrix Anal. Appl. \textbf{38} (2017), no.~2, 322--342.

\bibitem{ZimmermanHuper}
R.~Zimmermann and K.~H{\"u}per, \emph{Computing the {R}iemannian logarithm on the {S}tiefel manifold: Metrics, methods, and performance}, SIAM Journal on Matrix Analysis and Applications \textbf{43} (2022), 953–--980.

\end{thebibliography}
\end{document}